\documentclass[12pt]{amsart}
\usepackage{amsfonts}
\usepackage{fullpage, amsmath, amsthm,amsfonts,amssymb,stmaryrd, mathrsfs}
\usepackage{hyperref}

\usepackage[pdftex]{color}
\usepackage[all]{xy}

\newtheorem{theorem}[subsection]{Theorem}
\newtheorem{lemma}[subsection]{Lemma}
\newtheorem{corollary}[subsection]{Corollary}
\newtheorem{conjecture}[subsection]{Conjecture}

\newtheorem{proposition}[subsection]{Proposition}
\theoremstyle{definition}

\newtheorem{definition}[subsection]{Definition}
\newtheorem{definition-proposition}[subsection]{Definition-Proposition}
\newtheorem{hypothesis}[subsection]{Hypothesis}
\newtheorem{example}[subsection]{Example}

\newtheorem{remark}[subsection]{Remark}

\newtheorem{notation}[subsection]{Notation}

\numberwithin{equation}{subsection}

\def\calC{\mathcal{C}}

\def\calO{\mathcal{O}}
\def\calP{\mathcal{P}}

\def\calW{\mathcal{W}}

\def\gothm{\mathfrak{m}}

\def\gothU{\mathfrak{U}}

\def\gothX{\mathfrak{X}}

\def\bfB{\mathbf{B}}
\def\AAA{\mathbb{A}}

\def\CC{\mathbb{C}}
\def\FF{\mathbb{F}}
\def\GG{\mathbb{G}}

\def\NN{\mathbb{N}}

\def\QQ{\mathbb{Q}}
\def\RR{\mathbb{R}}

\def\ZZ{\mathbb{Z}}

\def\bfM{\mathbf{M}}

\def\rmM{\mathrm{M}}

\DeclareMathOperator{\Gal}{Gal}

\DeclareMathOperator{\Max}{Max}

\newcommand{\quash}[1]{}  
\newcommand{\an}{\mathrm{an}}

\newcommand{\Iw}{\mathrm{Iw}}
\newcommand{\unr}{\mathrm{unr}}
\newcommand{\ord}{\mathrm{ord}}

\newcommand{\rig}{\mathrm{rig}}

\newcommand{\wt}{\mathrm{wt}}

\newcommand{\inte}{\mathrm{int}}

\newcommand{\lan}{\mathrm{l.an}}
\newcommand{\alg}{\mathrm{alg}}

\newcommand{\LP}{\mathrm{LP}}

\DeclareMathOperator{\GL}{GL}

\DeclareMathOperator{\Char}{Char}
\DeclareMathOperator{\Spc}{Spc}

\DeclareMathOperator{\Ind}{Ind}

\begin{document}

\title{The eigencurve over the boundary of weight space}
\author{Ruochuan Liu}
\address{Ruochuan Liu, Beijing International Center for Mathematical Research, Peking University, 5 Yi He Yuan Road, Beijing, 100871, China. }
\email{liuruochuan@math.pku.edu.cn}
\author{Daqing Wan}
\address{Daqing Wan, Department of Mathematics, University of California, Irvine, 340 Rowland Hall, Department of Mathematics, Irvine, CA 92697.}
\email{dwan@math.uci.edu}
\author{Liang Xiao}
\address{Liang Xiao, Department of Mathematics, University of Connecticut, Storrs, 341 Mansfield Road, Unit 1009, Storrs, CT 06269-1009.}
\email{liang.xiao@uconn.edu}\date{\today}
\dedicatory{In memory of Professor Robert F. Coleman}
\begin{abstract}
We prove that the eigencurve associated to a definite quaternion algebra over $\QQ$ satisfies the following properties, as conjectured by Coleman--Mazur and Buzzard--Kilford: (a) over the boundary annuli of weight space, the eigencurve is a disjoint union of (countably) infinitely many connected components each finite and flat over the weight annuli, (b) the $U_p$-slopes of points on each fixed connected component are proportional to the $p$-adic valuations of the parameter on weight space, and (c) the sequence of the slope ratios form a union of finitely many arithmetic progressions with the same common difference.
In particular, as a point moves towards the boundary on an irreducible connected component of the eigencurve, the slope converges to zero.
\end{abstract}
\thanks{
R.L. is partially supported by NSFC-11571017 and the Recruitment Program of Global Experts of China.
D.W. was partially supported by Simons Fellowship.
L.X. is  partially supported by Simons Collaboration Grant \#278433 and NSF grant DMS-1502147.}
\maketitle

\setcounter{tocdepth}{1}
\tableofcontents

\section{Introduction}

\subsection{Coleman--Mazur--Buzzard--Kilford Conjecture}
\emph{Eigencurves} were introduced in the groundbreaking work of R. Coleman and B. Mazur \cite{coleman-mazur} to study the $p$-adic variation of modular forms.
Roughly speaking, they are rigid analytic curves that parameterize finite slope overconvergent normalized $p$-adic eigenforms, where the $q$-expansions of these overconvergent modular forms vary $p$-adically continuously.
The study of the eigencurves has led to great success, for example in M. Kisin's proof of the Fontaine--Mazur Conjecture for overconvergent modular forms \cite{kisin}.
While the arithmetic properties and the local geometry of the eigencurves were extensively studied in the literature (see e.g. \cite{bellaiche} for a summary), their global geometry seems to be a very intriguing and difficult topic.
Only recently, in joint work of H. Diao with the first author \cite{diao-liu}, they proved the ``properness" \footnote{Actually the "properness" here is not in the sense of rigid analytic geometry; we refer the reader to \cite{diao-liu} for its precise definition.} of the eigencurves over weight space.
In this paper, we focus on another interesting geometric property of eigencurves, namely, their behavior over the boundary annuli of weight space.

Let us be more precise.
Let $p$ be a prime number. Set $q=p$ if $p$ odd, and $q=4$ if $p=2$. Let $\varphi(q)$ denote the Euler function, namely, $\varphi(q) = p-1$ if $p$ is odd and $\varphi(q) = 2$ if $p=2$.
\quash{(excluding the case $p=2$ throughout the paper for a simple presentation).}
We use $v(\cdot)$ and $|\cdot|$ to denote the $p$-adic valuation and the $p$-adic norm, respectively, normalized so that $v(p) = 1$ and $|p| = p^{-1}$. In particular, $v(q) = 1$ if $q =p$, and $v(q) =2$ if $p=2$.
\emph{Weight space} $\calW$ is
the rigid analytic space associated to the Iwasawa algebra $\Lambda = \ZZ_p\llbracket \ZZ_p^\times \rrbracket$, which is the union of $\varphi(q)$ open unit disks indexed by the characters of the torsion subgroup $\Delta$ of $\ZZ_p^\times$.  Each closed point on weight space corresponds to a continuous ($p$-adic) character $\chi$ of $\ZZ_p^\times$.   We take the parameter on the weight disks to be $T:=T_{\chi}: = \chi(\exp(q))-1$.
For $r \in (0,1)$, we use $\calW^{> r}$ to denote the (union of) annuli where $|T| > r$; it is referred to as the ``\emph{halo}" of weight space by Coleman.

We fix a tame level and let $\calC$ denote the corresponding eigencurve, as constructed in K. Buzzard's paper \cite{buzzard} (which generalizes \cite{coleman-mazur}).
Each point of the eigencurve corresponds to a finite slope normalized overconvergent eigenform $f = \sum_{n \geq 0} a_n(f)q^n$.\footnote{This is the only time in this paper $q$ stands for $e^{2\pi iz}$. We will not mention $q$-expansions again.}
This eigencurve admits a map $\wt$ to the weight space, known as the \emph{weight map}, and a map $a_p$ to $\GG_m^\rig$, known as the \emph{slope map}.
\begin{equation}
\label{E:ap and wt}
\xymatrix{
\calC \ar[r]^-{a_p} \ar[d]^\wt & \GG_m^\rig \\
\calW
}
\end{equation}
For example, if we use $z_f$ to denote the point on $\calC$ corresponding to a classical normalized eigenform $f$ of weight $k+2 \geq 2$ and Nebentypus $p$-character $\chi
: (\ZZ/p^m\ZZ)^\times \to \CC_p^\times$, then the image of $z_f$ under the map $a_p$ is the $p$-th Fourier coefficient $a_p(f)$ of $f$, and the image of $z_f$ under the map $\wt$ is the point on $\calW$ corresponding to the character $x^k\chi: \ZZ_p^\times \to \CC_p^\times$ that sends $u$ to $u^k \chi(u)$.
In particular, the value of the parameter $T$ at this point is $T_{x^k \chi} =\exp(kq) \cdot \chi(\exp(kq)) -1 $.

For $r \in (0,1)$, we use $\calC^{> r}$ to denote the preimage $\wt^{-1}(\calW^{> r})$.
The following is a folklore conjecture suggested by a computation of  Buzzard and L. Kilford \cite{buzzard-kilford} which addresses a question asked by Coleman and Mazur \cite{coleman-mazur}.

\begin{conjecture}[Coleman--Mazur--Buzzard--Kilford]
\label{Conj:Buzzard-Kilford}
When $r \in (0,1)$ is  sufficiently close to $1^-$, the following statements hold.
\begin{enumerate}
\item
The space $\calC^{> r}$ is a disjoint union of (countably infinitely many) connected components $Z_1, Z_2, \dots$ such that the weight map $\wt: Z_n \to \calW^{> r}$ is finite and flat for each $n$.
\item
There exist nonnegative rational numbers $\alpha_1, \alpha_2, \ldots \in \QQ$ in non-decreasing order and tending to infinity such that, for each $n$ and each point $z \in Z_n$, we have
\[
|a_p(z)| = |T_{\mathrm{wt}(z)}|^{\alpha_n}.
\]
\item
The sequence $\alpha_1, \alpha_2, \dots$ is a disjoint union of finitely many arithmetic progressions, counted with multiplicity (at least when the indices are large enough).
\end{enumerate}

\end{conjecture}

When the tame level is trivial and $p=2$, this conjecture was verified using an explicit computation by Buzzard and Kilford \cite{buzzard-kilford}, extending the thesis work of M. Emerton \cite{emerton}.
More explicit computations for small $p$ and small tame levels have appeared in \cite{jacobs, kilford, kilford-mcmurdy,  roe}.
A partial result that is independent of the prime $p$ and the tame level was proved by J. Zhang and the second and third authors \cite{wan-xiao-zhang}.\\

The goal of this paper is to prove an analog of Conjecture~\ref{Conj:Buzzard-Kilford} for overconvergent automorphic forms for definite quaternion algebras over $\QQ$.

We fix some notations first. Let $D$ be a definite quaternion algebra over $\QQ$ which splits at $p$. Fix a tame level structure; we say the tame level is \emph{neat} if it satisfies the condition (Neat) in Subsection~\ref{SS:Buzzard's construction}.
\quash{(satisfying the condition (Neat) in Subsection~\ref{SS:Buzzard's construction}).}Let $\Spc_D$ denote the corresponding \emph{spectral curve} associated to the overconvergent automorphic forms for $D^\times$ constructed by Buzzard in \cite{buzzard}, which admits maps $\wt$ and $a_p$ similar to the eigencurve as in \eqref{E:ap and wt}. For $r\in(0,1)$, we denote by $\Spc_D^{>r}$ the preimage $\wt^{-1}(\calW^{>r})$. For each character $\omega:\Delta\rightarrow\ZZ_p^\times$, we denote by $\calW_{\omega}$ and $\calW_\omega^{>r}$ the weight disk and annulus corresponding to $\omega$, and by $\Spc_{D,\omega}$ and $\Spc_{D,\omega}^{>r}$ the preimages $\wt^{-1}(\calW_{\omega})$ and $\wt^{-1}(\calW^{>r}_{\omega})$ respectively.
Let $\omega_0: \Delta \to \ZZ_p^\times$ denote the inclusion map. 

The first main result is the following theorem, in which 
\begin{itemize}
\item
the constant $t$ is equal to the dimension of the space of weight $2$ automorphic forms with $q$-Iwahori level structure at $p$ and tame level as above, and
\item
$r_{\ord}(\omega)$ denotes the dimension of the ordinary subspace of automorphic forms of weight 2 and character $\omega$.
\end{itemize}

\begin{theorem}
\label{T:main theorem}
Let $\omega: \Delta\to \ZZ_p^\times$ be a character. Then the space $\Spc_D^{>1/p}$ is a disjoint union 
\[
\Spc_D^{>1/p} = X_0 \coprod X_{(0,1)} \coprod X_1 \coprod X_{(1,2)} \coprod X_2 \coprod \cdots
\]
of (possibly empty) rigid analytic spaces which are finite and flat over $\calW^{>1/p}$ via $\wt$,
such that, for each point $x \in X_I$ with $I$ denoting the interval $n=[n, n]$ or $(n,n+1)$, we have
\[
v(a_p(x)) \in \varphi(q)v(T_{\wt(x)})\cdot I.
\]
In particular, as $x$ varies on each irreducible component of $\Spc_D$ with $\wt(x)$ approaching the boundary of weight space, i.e. 
$|T_{\wt(x)}| \to 1^-$, the slope $v(a_p(x)) \to 0$.

Moreover, if the tame level is neat, and we denote by $X_{n,\omega}$ and $X_{(n,n+1),\omega}$ the preimages of $\calW_{\omega}^{>1/p}$
in $X_n$ and $X_{(n,n+1)}$ respectively, then
\[
\deg X_{n,\omega} = \left\{
\begin{array}{ll}
r_{\ord}(\omega),&\textrm{if } n=0,\\
r_{\ord}(\omega^{-1}\omega_0^{2n-2})+r_{\ord}(\omega\omega_0^{-2n}),& \textrm{if }n\geq1,\\
\end{array}
\right.
\]
and
\[
\deg X_{(n,n+1),\omega} =qt-r_{\ord}(\omega^{-1}\omega_0^{2n})-r_{\ord}(\omega\omega_0^{-2n}).
\]
 for all $n\geq0$. In particular, we have $\deg X_{(n,n+1),\omega}>0$ for all $n\geq0$.
\end{theorem}

\quash{\begin{remark}
When $p=2$ and the tame level is neat, we expect the optimal $r$ in the above theorem to be $\frac{1}{4}$. See also Remark \ref{R:remark of main theorem}.
\end{remark}}
This theorem will be proved in Subsection~\ref{SS:Buzzard-Kilford D}.

\begin{corollary}
\label{C:periodicity}
If the tame level is neat, then for $I=(0,1), 1, (1,2), 2,\dots$, we have 
\[
\deg X_{I,\omega}=\deg X_{ I+1,\omega\omega_0^2}.
\]  
In particular, the degree $\deg X_{I,\omega}$ is periodic modulo $\frac{\varphi(q)}{2}$.
\end{corollary}

In order to prove the full version of Conjecture \ref{Conj:Buzzard-Kilford} for $\Spc_D$, our current technique requires to weaken the radius bound on $|T|$.
The following theorem will be proved in Subsection~\ref{SS:proof of main theorem 2}.

 \begin{theorem}
\label{T:main theorem 2}
Let $\omega: \Delta\to \ZZ_p^\times$ be a character. Then there exists $\lambda\in(0,1)$ such that  
there exists a sequence of rational numbers $\alpha_0(\omega), \alpha_1(\omega), \dots$ in increasing order and tending to infinity such that
$\Spc_{D, \omega}^{> \lambda}$ is a disjoint union $\coprod_{i \geq 0} Y_{i,\omega}$ of rigid analytic spaces finite and flat over $\calW^{> \lambda}_\omega$ via $\wt$, such that
\begin{equation}
\label{E:slope ratio equality}
v(a_p(y)) = \varphi(q)v(T_{\wt(y)})\alpha_i(\omega)
\end{equation}
for every $y\in Y_{i,\omega}$. More precisely, if the tame level is neat, then we can take $\lambda=p^{-\frac{8}{(p^2-1)t+8}}$ for $p>2$, and $\lambda=2^{-\frac{1}{t+1}}$ for $p=2$.

Moreover, let $M$ be a positive integer so that $p^{-q/p^{M-1}(p-1)} >\lambda$; we also require $M \geq 2$ if $p$ is odd and $M\geq4$ if $p=2$.
Then if we extend the sequence $\alpha_0(\omega), \alpha_1(\omega), \dots$ into $\tilde \alpha_0(\omega), \tilde \alpha_1(\omega), \dots$ with each $\alpha_i(\omega)$ appearing with the multiplicity $\deg Y_{i, \omega}$, then $\tilde \alpha_0(\omega),\tilde \alpha_1(\omega) \dots$ is a disjoint union of $\frac{(p-1)p^{M-1}t}{2}$ arithmetic progressions with (the same) common difference $\frac{\varphi(q)p^M}{2q^2}$.
More precisely, we have
\begin{equation}
\label{E:finer twisted periodicity}
\tilde\alpha_{j+q^{-1}p^{M}t}(\omega \omega_0^2)=\tilde\alpha_{j}(\omega)+\tfrac{p^M}{q^2}\quad \textrm{for any }j \geq 0.
\end{equation}

\end{theorem}

\begin{remark}
\label{R:remark of main theorem}
We first remark on the content of the theorems.
\begin{enumerate}
\item
It is implicit from the statements that, for each $X_{I,\omega}$ from Theorem~\ref{T:main theorem}, $X_{I,\omega} \times_{\calW_\omega^{>1/p}} \calW_\omega^{> \lambda}$ is the disjoint union of those $Y_{i,\omega}$ in Theorem~\ref{T:main theorem 2} for which $\alpha_i(\omega) \in I$.

\item
The bound given by Theorem~\ref{T:main theorem 2} appears to depend heavily on $t$. It might be possible to release $t$ to $t_{\bar \rho}$ by working with each residual pseudo-representation $\bar \rho$, where $t_{\bar \rho}$ denotes the dimension of the space of weight $2$ automorphic forms with Iwahori level structure at $p$ where the tame Hecke action is determined by $\bar \rho$.
More generally, we expect our argument to continue to hold for a direct summand of the completed homology of a modular curve or a definite quaternion algebra. We will revisit this idea in a future work.

\quash{\item
We expect the theorem to hold  for $p=2$ with some minor modification.}

\item For a continuous character $\chi$ of $\ZZ_p^\times$, the $p$-adic valuation of $T_\chi$ is the same as the $p$-adic valuation of $\chi(c) -1$ for any topological generator $c$ of $(1+q\ZZ_p)^\times$.  Therefore, both theorems do not depend on our convenient choice of the generator $\exp(q)$. Moreover, the region $|T|>1/p$ is stable under the change of variable $T \mapsto \exp(-kq)(T+1)-1$ that recenters the weight disks around the classical weight $x^k$.

\item
The radius $1/p$ of Theorem~\ref{T:main theorem} seems to be optimal if $p>2$. When $p=2$, one might be able to improve the radius to $1/4$ as opposed to the $1/2$ given in Theorem~\ref{T:main theorem}.
The estimate of radius $\lambda$ and the positive integer $M$ in  Theorem~\ref{T:main theorem 2} may not be optimal. We do not know whether it is reasonable to expect any optimal bound.
\item
The proof of Theorem~\ref{T:main theorem} gives rise to a certain integral model of the spectral curve near the boundary of weight space, by factoring the characteristic power series of $U_p$ integrally. See Remark~\ref{R:integral model of spectral curve} for an elaborated discussion.

\item
The difference between the spectral curve and the actual eigencurve is minor for the type of questions we consider in this paper, as the eigencurve is essentially a (partial) normalization of the spectral curve (and possibly changing some non-reduced structure).
\end{enumerate}

\end{remark}

\begin{remark}
\label{R:relation to others work}
We remark on the relation to the literature.
\begin{enumerate}
\item
By G. Chenevier's $p$-adic Jacquet--Langlands correspondence \cite{chenevier}, we can translate results from the case of automorphic forms for definite quaternion algebras to the case of modular forms, and hence prove a large portion of Conjecture~\ref{Conj:Buzzard-Kilford}.
The only connected components of the eigencurve we cannot access by this method are the ones whose tame part are all principal series. However, see Remark~\ref{R:generalization}(2) for a discussion of potential approaches to this case.
\item
While the main result of \cite{wan-xiao-zhang} has now become a corollary of our two main theorems (See Corollary~\ref{C:classical}), our proof relies on several ideas developed therein.
The key improvement from \cite{wan-xiao-zhang} is that we choose a better basis to estimate the Newton polygon.
See Remark~\ref{R:Buzzard freedom} for a more detailed discussion.

\item
Some related results were proved by F. Andreatta, A. Iovita, and V. Pilloni for the usual overconvergent (Hilbert) modular forms \cite{AIP} in the sense of Coleman--Mazur \cite{coleman-mazur} and Andreatta--Iovita--Pilloni--Stevens \cite{AIS, pilloni}; they constructed a certain compactification of weight space in the category of analytic adic spaces, and showed that the sheaf of overconvergent modular forms extends. They also obtained certain results on the geometry of the eigencurve near the boundary of weight space. Although their technique appears different to ours, both works have the same goal: realizing Coleman's idea \cite{coleman-halo} in the corresponding context. So it would be interesting to compare the two approaches.

\item
The second half of Theorem~\ref{T:main theorem 2} follows from the first half by classicality results and Atkin--Lehner theory. This argument was independently found by J. Bergdall and R. Pollack \cite{bergdall-pollack}.

\item
Corollary~\ref{C:periodicity} indicates a peculiar relation between the degrees of components of the spectral curve over one weight disk and those over another weight disk shifted by the square of the Teichm\"uller character. This might be related to some observation by F. Calegari, known as the ``theta-cycle" phenomenon: in characteristic $p$, there is a $\theta$ map on the space of mod $p$ modular forms, increasing the weight by $2$; this seems to have some magical effect on the slope of modular forms (which are of characteristic zero).

\item
In this paper, we do not touch the geometry of the eigencurves over the center of weight space, which is expected to be very complicated.  We refer to \cite{wan, buzzard-question, lisa-clay, buzzard-gee, bergdall-pollack2} for a more comprehensive discussion.
\end{enumerate}
\end{remark}

We now turn to discussing the application of our main theorem.
For a character $\psi:(\ZZ/p^m\ZZ)^\times \to \CC_p^\times$ that does not factor through $(\ZZ/p^{m-1}\ZZ)^\times$ (which we call \emph{$p$-primitive}) and $k\in \ZZ_{\geq 0}$, we use $S_{k+2}^D(\psi)$ to denote the space of automorphic forms on $D^\times$ of weight $k+2$ with the fixed tame level, $p^m$-Iwahori level at $p$, and  Nebentypus character $\psi$.
Combining Theorems~\ref{T:main theorem} and \ref{T:main theorem 2} with the classicality result (Proposition~\ref{P:classicality}), one can deduce strong consequences regarding the slopes of classical automorphic forms.  
Roughly speaking, we prove that \emph{knowing the slopes of weight $2$ automorphic forms of $\varphi(q)$ characters with a certain conductor at $p$ is enough to determine the slopes of all automorphic forms with larger conductors at $p$.}
The precise statement is as follows, whose proof will appear in Subsection~\ref{S:proof of corollary}.
\begin{corollary}
\label{C:classical}
\begin{enumerate}
\item
Let $\psi$ be a $p$-primitive character of $(\ZZ/p^m\ZZ)^\times$ with $m \geq 2$ if $p>2$, and $m\geq 4$ if $p=2$.
Let $\beta_0(k, \psi), \dots, \beta_{q^{-1}p^{m}(k+1)t-1}(k, \psi)$ denote the sequence of slopes of the $U_p$-action on $S_{k+2}^D(\psi)$, in non-decreasing order and counted with multiplicity.
Then we have
\begin{equation}
\label{E:slopes rough bound}
\frac{q^2}{p^m}\big(\lfloor n/qt\rfloor \big)\leq
\beta_n(k, \psi) \leq \frac{q^2}{p^m}\big(\lfloor n/qt  \rfloor +1 \big),
\end{equation}
for $n =0, \dots, q^{-1}p^{m}(k+1)t-1$.
Note the inequalities we obtained are \emph{independent} of the weight $k+2$.

\item
Let $M$ be a positive integer so that $p^{-q/p^{M-1}(p-1)}>\lambda$; we also require $M \geq 2$ if $p>2$ and $M\geq4$ if $p=2$.
For each character $\omega$ of $\Delta$, we choose a $p$-primitive character $\psi$ of $(\ZZ/p^{M}\ZZ)^\times$ as above so that $\psi|_\Delta = \omega$, and 
 let $\beta_0(\omega), \dots, \beta_{q^{-1}p^{M}t-1}(\omega)$ denote the sequence of slopes of the $U_p$-action on $S_{2}^D(\psi)$, in non-decreasing order and counted with multiplicity.
(This sequence does not depend on the choice of $\psi$.)
Then for any $k \in \ZZ_{\geq 0}$ and any $p$-primitive character $\psi_m$ of $ (\ZZ/p^m\ZZ)^\times $ with $m \geq M$, the sequence of the slopes of the $U_p$-action on $S^D_{k+2}(\psi_m)$ is given by
\[
\bigcup_{n=0, \dots,p^{m-M}(k+1)-1}
 \Big\{p^{M-m}\big(\beta_0(\psi_m|_\Delta  \omega_0^{k-2n})+n\big), \dots, p^{M-m}\big(\beta_{q^{-1}p^{M}t-1}(\psi_m|_\Delta  \omega_0^{k-2n})+n\big) \Big\}.
\]
\end{enumerate}
\end{corollary}

Another application is the following result.
\begin{corollary}
\label{C:component has classical points}
Each irreducible component of the spectral curve $\Spc_D$ contains a classical point of weight $2$ (with possibly large conductor at $p$).
\end{corollary}
\begin{proof}
By \cite[Corollary~1.3.13]{coleman-mazur}, each irreducible component extends to the boundary of  weight space, where the slopes tend to zero by Theorem~\ref{T:main theorem}.
So this irreducible component must contain a point above a classical weight $2$ (of large conductor at $p$) and with slope strictly less than $1$. This point must correspond to a classical automorphic form by the classicality result (see \cite[Proposition~4]{buzzard} or Proposition~\ref{P:classicality}).
\end{proof}

\begin{remark}
This corollary naturally appeared in the joint work of J. Pottharst and the third author \cite{pottharst-xiao}.
In that paper, we studied the parity conjecture which predicts that the vanishing orders of the L-functions of modular forms are congruent modulo $2$ to the dimension of the associated Selmer groups.
The basic idea is that, one can prove this in weight $2$ by applying some argument involving Heegner points.
The main theorem of \cite{pottharst-xiao} is roughly that, on each irreducible component of the eigencurve, if the parity conjecture holds for one classical point, then it holds for all classical points. So by Corollary~\ref{C:component has classical points}, any modular form, if can be translated to an automorphic form over a definite quaternion algebra (split at $p$), is linked to a classical automorphic form of weight $2$ (of slope $<1$). Hence the parity conjecture holds for that modular form.

We expect a similar argument can be applied to study the $p$-adic Gross--Zagier formula, to bypass the essential difficulty imposed by requiring slopes $<1$.
\end{remark}



\subsection{Idea of the proof of Theorems~\ref{T:main theorem} and \ref{T:main theorem 2}}
We point out a few key points in the proof of the main theorems.
\begin{enumerate}

\item
In Coleman's private note \cite{coleman-halo}, he advocated the idea of viewing the weight space and the eigencurve as formal schemes, as opposed to (increasing unions of) rigid analytic spaces.
He pointed out that the key to realize this is to provide a certain ``integral model" of the space of overconvergent modular forms \emph{over the ``halo" of weight space}.
Although we shall be working with a context different from what he suggested in \cite{coleman-halo}, 
\emph{this viewpoint is absolutely crucial to our paper.}
In the case for definite quaternion algebra we study in this paper, Coleman's idea amounts to construct a ``Banach space" over  the Iwasawa algebra $\Lambda$, whose base change to each affinoid subdomain of $\calW$ is ``close to" the Banach space of overconvergent automorphic forms in the sense of Buzzard \cite{buzzard}, at least having the same characteristic power series for $U_p$.
In fact, this expected space is not mysterious: its dual is  the coinvariant subspace of Emerton's completed homology under the action of the unipotent radical of the Borel subgroup at $p$, which is a compact topological $\Lambda$-module (in the sense of \cite{schneider-teitelbaum}).
In this paper, we present the construction using induced representations; this gives rise to a ``Banach $\Lambda$-module", which we call the space of \emph{integral $p$-adic automorphic forms}. (The action is slightly twisted to match with the convention used by Buzzard \cite{buzzard2}.)
We refer to Remark~\ref{R:relation to Emerton} for the relation with Emerton's completed homology and potential generalizations.

\quash{One caveat is the following: the $U_p$-action on the space of integral automorphic forms may \emph{not} be compact (Remark~\ref{R:Up not compact}); in particular, the characteristic power series of $U_p$, although it can still be defined for the explicit orthonormal basis we care about, would potentially depend on the choice of an orthonormal basis.
This annoying issue does not affect the proof of Theorem~\ref{T:main theorem}, but we feel that giving a good integral model of the space of overconvergent modular forms is very important.
For this, 
we construct in Section~\ref{Sec:integral model} a variant of this space, which is properly rescaled on an explicit orthonormal basis.
This space is only defined over $\Lambda^{>1/p}$; we do not know if one can extend the definition to over $\Lambda$.
We hope this space may be useful for the study of arithmetic of the eigencurve near the boundary of weight space.}

\item
We choose to work with a definite quaternion algebra as opposed to the usual overconvergent modular forms, to circumvent the complication of the geometry of the modular curves, as presented in all prior works of direct computation (they all rely on the explicit equation that defines the modular curve, which is clearly inaccessible in general).
In our case,  the $U_p$-action on the space of integral $p$-adic automorphic forms can be written reasonably explicitly, as explained in the first part of Section~\ref{Sec:Estimation of NP}.  This was inspired by the thesis of D. Jacobs \cite{jacobs} (a former student of Buzzard), and our generalization \cite{wan-xiao-zhang}.

\item
Using the explicit description of the space of integral $p$-adic automorphic forms, we look at the associated infinite matrix $(P_{i,j})$ with respect to a basis originated from the Mahler basis $1, z, \binom z2, \dots$ on the space of $p$-adic continuous functions on $\ZZ_p$.
A mild $p$-adic analysis computation (which is the core of our paper) shows that $P_{i,j} \in \gothm_\Lambda^{\max\{0, \lfloor i/t\rfloor - \lfloor j / pt\rfloor\}}$, where $\gothm_\Lambda$ is the ideal of $\Lambda$ generated by $p$ and $T$.
As a consequence, if we write $c_0 + c_1X + \cdots \in\Lambda\llbracket X \rrbracket$ for the characteristic power series for the $U_p$-operator, then $c_i$ belongs to \[
T^{\lambda_i} \Lambda^{>1/p},
\]
where $\lambda_i$ is recursively defined by $\lambda_0 = 0$, and $\lambda_i-\lambda_{i-1} = \lfloor i/t \rfloor - \lfloor i/pt \rfloor$.
This gives rise to a lower bound on the Newton polygon over each point of weight space with $|T|>1/p$.

\item
It is somewhat a lucky coincidence that, {\it the Newton polygon lower bound obtained in (3) partially agrees with the actual Newton polygon at classical weights}.
This allows us to conclude the main theorems.  This part of the argument was inspired by similar tricks used in joint work of the last two authors with C. Davis \cite{davis-wan-xiao}.
\end{enumerate}

\subsection{Structure of the paper}

Section~\ref{Sec:automorphic forms} is devoted to constructing a certain integral model for the space of $p$-adic automorphic forms on a definite quaternion algebra.
The action of $U_p$-operator on this integral model was made explicit in the first part of Section~\ref{Sec:Estimation of NP}; and
we prove Theorem~\ref{T:main theorem} in the latter part of Section~\ref{Sec:Estimation of NP} using a close estimate of the Newton polygon.
Section~\ref{Sec:distribution of Up} is devoted to proving Theorem~\ref{T:main theorem 2}.
In Section~\ref{Sec:integral model}, we provide a variant of the construction given in Section~\ref{Sec:automorphic forms},  which can be regarded as integral models of the space of overconvergent automorphic forms.

\subsection{Acknowledgments}
We cannot emphasize enough the importance of the ideas of Robert Coleman to this paper.
We thank Barry Mazur for his constant encouragement and many suggestions. We
thank Vincent Pilloni for sharing his insight into Coleman's idea.
We thank the anonymous referees for their impressively helpful report which greatly improved the exposition of the paper as well as simplified some arguments.
We thank John Bergdall, Gaetan Chenevier, Keith Conrad, and Robert Pollack for interesting discussions.
D.W. and L.X. thank the hospitality of Beijing International Center for Mathematical Research when they visited.

\subsection{Notation}
Throughout this paper, $\mathbb{N}$ denotes the set of \emph{positive} integers. We fix a prime number $p$. Set $q=p$ for $p>2$, and $q=4$ for $p=2$.  Let $\varphi(q)$ denote the Euler function, namely, $\varphi(q) =p-1$ if $p>2$ and $\varphi(q) = 2$ if $p=2$. Write $\AAA$ for the ring of adeles of $\QQ$, and $\AAA_f$ (resp. $\AAA_f^{(p)}$) the subring of finite adeles (resp. finite prime-to-$p$ adeles).

For $A$ an affinoid $\QQ_p$-algebra, we use $A^\circ$ to denote the subring of power bounded elements.
The notions $A\langle z \rangle$ and $A^\circ\langle z \rangle$ are reserved for denoting Tate algebras.

\emph{The row and column indices of matrices always start with $0$.}
We use $I_n$ for $n \in \ZZ_{\geq 0}$ or $\infty$ to denote the identity $n \times n$-matrix.

\section{Automorphic forms for definite quaternion algebras}
\label{Sec:automorphic forms}
We first discuss carefully various versions of (overconvergent) automorphic forms for definite quaternion algebras. In particular, we give a certain ``integral model" of the space of $p$-adic automorphic forms.

\begin{notation}
 We write $\ZZ_p^\times $ as $\Delta \times (1+q\ZZ_p)^\times$ with $\Delta \cong (\ZZ/q\ZZ)^\times$. We identify $(1+q\ZZ_p)^\times $ with $ \ZZ_p$ via $\frac 1q \log(-)$. 
Let $\Lambda$ denote the Iwasawa algebra 
\[
\ZZ_p\llbracket\ZZ_p^\times\rrbracket \cong \ZZ_p[\Delta] \otimes_{\ZZ_p} \ZZ_p\llbracket(1+q\ZZ_p)^\times\rrbracket \cong \ZZ_p[\Delta] \otimes_{\ZZ_p} \ZZ_p\llbracket T\rrbracket,
\]
where $T$ corresponds to $[\exp(q)]-1$.  Here, for $a \in \ZZ_p^\times$, we use $[a]$ to denote its image in $\Lambda^\times$; so $[-]: \ZZ_p^\times \to \Lambda^\times$ is the \emph{universal character} of $\ZZ_p^\times$.
In particular, each continuous ring homomorphism $\chi: \Lambda \to \CC_p$ defines a continuous character $\chi\circ [-]: \ZZ_p^\times \to \CC_p^\times$ (which we still denote by $\chi$).
Conversely, all continuous $\CC_p$-valued characters of $\ZZ_p^\times $ may be obtained this way.

We use 
$\gothm_\Lambda$ to denote the ideal $(p, T)$ of $\Lambda \cong \ZZ_p[\Delta] \otimes_{\ZZ_p} \ZZ_p\llbracket T\rrbracket$.

Let $\calW$ denote the rigid analytic space associated to $\Lambda$.
For each ($\CC_p$-valued) continuous character $\chi$ of $\ZZ_p^\times$, we write $T_\chi: = \chi(\exp(q))-1$ for the \emph{$T$-coordinate} of the associated point on weight space.
For $r \in (0,1)\cap p^\QQ$, we use $\calW^{\leq r}$ to denote the union of the disks where $|T|\leq r$; it is an affinoid subdomain of the weight space.

Following Buzzard \cite[Section 5]{buzzard2}, we  define for $m \in \NN$ the rigid analytic spaces
\begin{align*}
\bfB_{p^{-m}} &= \big\{z \in \calO_{\CC_p}\; \big|\; |z-a| \leq p^{-m} \textrm{ for some }a \in \ZZ_p\big\};\\
\bfB_{p^{-m}} ^\times &= \big\{z \in \calO_{\CC_p}\; \big|\; |z-a| \leq p^{-m} \textrm{ for some }a \in \ZZ_p^\times\big\}.
\end{align*}

For $m \in \NN$ ($m \geq 2$ if $p=2$), we say a continuous character $\chi: \ZZ_p^\times \to A^\times$ with values in an affinoid $\QQ_p$-algebra $A$ is \emph{$m$-locally analytic} if  for each closed point $x \in \Max(A)$ and the corresponding character $\chi_x: \ZZ_p^\times \xrightarrow{ \chi}A^\times \to k(x)^\times$, we have $v(T_{\chi_x}) > q/p^{m}(p-1)$.
In this case,
$\chi$ extends to a continuous homomorphism 
\begin{align*}
\kappa: (\ZZ_p + p^m A^\circ\langle z \rangle)^\times = \ZZ_p^\times \cdot (1+p^m& A^\circ\langle z \rangle) \longrightarrow (A^\circ\langle z \rangle)^\times
\\
a\cdot x &\longmapsto \chi(a) \cdot \chi(\exp(p^m))^{(\log x)/{p^m}}.
\end{align*}
When $A$ is a finite extension $E$ of $\QQ_p$, this means that $\chi$ extends to a homomorphism of rigid group schemes $\chi: \bfB^\times_{p^{-m}} \to \GG_{m, E}^\rig$.
See e.g. \cite[(3.1.2)]{wan-xiao-zhang} for more discussion.

For $m \in \NN$, a continuous character $\psi: \ZZ_p^\times \to E^\times$ (with $E$ a finite extension over $\QQ_p$) is called a \emph{finite character of conductor $p^m$} if it factors through $(\ZZ/p^m\ZZ)^\times$ but not $(\ZZ/p^{m-1}\ZZ)^\times$.
We say a continuous character $\chi$ of $\ZZ_p^\times$ is \emph{classical} if it sends $x$ to $x^k \psi(x)$ for an integer $k \geq 0$ and a finite character $\psi$ of conductor $p^m$. We write $(k,\psi)$ for such a character; it is $m$-locally analytic, because
\begin{itemize}
\item
(when $p>2$) $v(T_{(k, \psi)})\geq1$ if $m=1$, and $v(T_{(k, \psi)})= 1/p^{m-2}(p-1)$ if $m\geq 2$, and
\item
(when $p=2$) $v(T_{(k, \psi)})\geq1$ if $m = 3$, and $v(T_{(k, \psi)})= 1/p^{m-3}$ if $m\geq 4$.
\end{itemize}  By abuse of language, we say this $(k,\psi)$ has \emph{conductor $p^m$}. In this paper, the weight of automorphic forms will be $k+2$.
\end{notation}

\subsection{Subgroups of $\GL_2(\ZZ_p)$}
\label{SS:subgroups of GL2(Zp)}

We consider the following subgroups of $\GL_2(\QQ_p)$ (for $m \in \NN$):
\[
\Iw_{p^m}\cong 
\begin{pmatrix} \ZZ_p^\times & \ZZ_p \\ p^m\ZZ_p & \ZZ_p^\times
\end{pmatrix} \ \supset \ 
B(\ZZ_p)\cong 
\begin{pmatrix} \ZZ_p^\times & \ZZ_p \\ 0 & \ZZ_p^\times
\end{pmatrix} \ \supset \ 
N(\ZZ_p)\cong 
\begin{pmatrix}1 & \ZZ_p \\ 0 & 1
\end{pmatrix},
\]
\[
T(\ZZ_p) = \begin{pmatrix}
\ZZ_p^\times &0\\0&\ZZ_p^\times
\end{pmatrix} \quad \textrm{and} \quad \bar N(p^m\ZZ_p)\cong 
\begin{pmatrix}1 & 0\\p^m\ZZ_p  & 1
\end{pmatrix}.
\]
The \emph{Iwasawa decomposition} is the isomorphism
\begin{equation}
\label{E:Iwasawa decomposition}
\xymatrix@R=0pt{
N(\ZZ_p) \times T(\ZZ_p) \times \bar N(p^m\ZZ_p) \ar[rr]^-\cong &&\Iw_{p^m}
\\
( n, t,\bar n) \ar@{|->}[rr]&&  n t\bar n.
}
\end{equation}
We will often identify $\bar N(q\ZZ_p)$ with $\ZZ_p$ by sending $\big( \begin{smallmatrix}
1&0\\qz&1
\end{smallmatrix} \big)$ to $z$ for $z\in \ZZ_p$.
The Iwahori subgroup $\Iw_{q}$ admits an anti-involution:
\begin{equation}
\label{E:anti-convolution}
g = \big( \begin{smallmatrix}
a&b\\ c&d
\end{smallmatrix} \big) \mapsto  g^*: = \big( \begin{smallmatrix}
1&0\\0&q
\end{smallmatrix} \big) g^\mathrm{t} \big( \begin{smallmatrix}
1&0\\0&q^{-1}
\end{smallmatrix} \big) = \big( \begin{smallmatrix}
a&c/q\\ qb&d
\end{smallmatrix} \big).
\end{equation}
In particular, $(gh)^* = h^*g^*$ for $g,h \in \Iw_{q}$.

Consider the natural homomorphism $\pi_d: B(\ZZ_p) \to \ZZ_p^\times$ sending $\big(\begin{smallmatrix}
a & b\\0&d
\end{smallmatrix} \big) $ to $d$.
Any continuous character of $\ZZ_p^\times$ can be viewed as a character of $B(\ZZ_p)$ by composing with $\pi_d$.
We will only consider characters of $B(\ZZ_p)$ of this form.

\subsection{Induced representation}
\label{SS:induced representation}
Let $A$ be a topological ring in which $p$ is topologically nilpotent; and let $\chi:\ZZ_p^\times \to A^\times$ be a continuous character, viewed as a character of $B(\ZZ_p)$ by composing with $\pi_d$ as above.
Consider the following induced representation of $\Iw_q$:
\[
\Ind_{B(\ZZ_p)}^{\Iw_q}(\chi): = \big\{ \textrm{continuous functions }f: \Iw_q \to A \;|\; f(bg) = \chi(b) f(g) \textrm{ for }b \in B(\ZZ_p),\; g \in \Iw_q\big\},
\]
where instead of the usual left action, we consider the right action of $h \in \Iw_q$ by sending $f$ to $f|\!|^\chi_h: g \mapsto f(gh^{\mathrm *})$. (The reason for our choice is to match with the convention used in Buzzard \cite{buzzard2} as shown by \eqref{E:right action on C(N)} below.)
The Iwasawa decomposition~\eqref{E:Iwasawa decomposition} gives the following isomorphism, which made this induced representation explicit:
\begin{equation}
\label{E:explicit induced representation}
\xymatrix@R=0pt{
\Ind_{B(\ZZ_p)}^{\Iw_q}(\chi) \ar[r]^-\cong & \calC (\ZZ_p; A): = \big\{ \textrm{continuous functions } \ZZ_p \to A\big\}
\\
f \ar@{|->}[r] & h(z) := f\Big({ \begin{pmatrix}
1&0\\ qz &1
\end{pmatrix}}
\Big).
}
\end{equation}
Then the (right) action of $\big(\begin{smallmatrix}
a&b\\c&d
\end{smallmatrix} \big) \in \Iw_q$
on $\Ind_{B(\ZZ_p)}^{\Iw_q}(\chi)$ induces the following action on $\calC (\ZZ_p; A)$:
\begin{equation}
\label{E:right action on C(N)}
h \big|\!\big|^\chi_{\big( \begin{smallmatrix}
a&b\\c&d
\end{smallmatrix} \big) } (z) = f \Big( \big( \begin{smallmatrix}
1&0\\qz&1
\end{smallmatrix} \big)\big( \begin{smallmatrix}
a& c/q\\qb&d
\end{smallmatrix} \big) \Big) = f \Big( \big( \begin{smallmatrix}
\frac{ad-bc}{cz+d}&c/q\\0&cz+d
\end{smallmatrix} \big)\big( \begin{smallmatrix}
1&0\\q\frac{az+b}{cz+d}&1
\end{smallmatrix} \big) \Big)= \chi(cz+d) h\Big( \frac{az+b}{cz+d} \Big).
\end{equation}
One checks that this action extends to an action of the monoid 
\begin{equation}
\label{E:monoid M1}
\bfM_1
: = \big\{ \big( \begin{smallmatrix}
a&b \\c&d
\end{smallmatrix} \big) \in \rmM_2(\ZZ_p) \; |\; 
q|c, \ p\nmid d, \textrm{ and }ad-bc \neq 0
\big\}.
\end{equation}

When $A$ is an affinoid $\QQ_p$-algebra and $\chi: \ZZ_p^\times \to A^\times$ is $m_0$-locally analytic for some $m_0\in \NN$, for every $m\geq \max\{m_0, v(q)\}$, we can consider the \emph{$m$-locally analytic induced representation} 
\[
\Ind_{B(\ZZ_p)}^{\Iw_q}(\chi)^{m,\an} = \Big\{ f \in \Ind_{B(\ZZ_p)}^{\Iw_q}(\chi)\; \big|\; f \textrm{ is an analytic function on }\big(\begin{smallmatrix} 1 & 0 \\ a + p^{m} \ZZ_p& 1\end{smallmatrix} \big),\textrm{ for all } a \in q\ZZ_p\Big\}.
\]
Here analytic function means that the values of the function can be given by a convergent Taylor series on the specified $p$-adic disk. The condition that $\chi$ is $m$-locally analytic is used so that the action \eqref{E:right action on C(N)} is well defined.

Similar to \eqref{E:explicit induced representation}, sending $f$ to $h(z) = f\big(\big( \begin{smallmatrix}
1&0\\qz&1
\end{smallmatrix} \big)\big)$ induces an isomorphism 
\begin{equation}
\label{E:m-loc analytic induced representation}
\Ind_{B(\ZZ_p)}^{\Iw_q}(\chi)^{m,\an} \xrightarrow{\ \cong \ } \calO_{\bfB_{qp^{-m}}} \widehat \otimes_{\QQ_p} A.
\end{equation}
Here the latter space may be understood as the subspace of continuous functions $h \in \calC(\ZZ_p; A)$ such that $h$ is analytic on $a+q^{-1}p^m\ZZ_p$ for all $a \in \ZZ_p$.

We write 
\[
\Ind_{B(\ZZ_p)}^{\Iw_q}(\chi)^{\lan} = \bigcup_{m \geq m_0} \Ind_{B(\ZZ_p)}^{\Iw_q}(\chi)^{m, \an}
\]
for the \emph{locally analytic induced representation}; it is a subspace of $\Ind_{B(\ZZ_p)}^{\Iw_q}(\chi)$ which, under the isomorphism \eqref{E:explicit induced representation}, corresponds to $\calC^\an(\ZZ_p; A)$ consisting of locally analytic functions on $\ZZ_p$ with values in $A$.

When $\chi$ is a classical character $(k, \psi)$ of conductor $p^{m_0} \leq p^m$, $\Ind_{B(\ZZ_p)}^{\Iw_q}(\chi)^{m,\an}$ contains a subspace $\Ind_{B(\ZZ_p)}^{\Iw_q}(\chi)^{m,\alg}$ consisting of those functions
whose restriction to $a+q^{-1}p^{m}\ZZ_p$ is a polynomial function of degree $\leq k$  for all $a \in \ZZ_p$. 
The isomorphism \eqref{E:m-loc analytic induced representation} induces an isomorphism between $\Ind_{B(\ZZ_p)}^{\Iw_q}(\chi)^{m,\alg}$ and the space $
\LP^{m-v(q), \deg \leq k}(\ZZ_p;  E)$ consisting of all functions $h \in \calC (\ZZ_p; E)$ whose restriction to $a+q^{-1}p^{m}\ZZ_p$ is a polynomial function of degree $\leq k$  for all $a \in \ZZ_p$.

\subsection{Buzzard's overconvergent automorphic forms}
\label{SS:Buzzard's construction}
We now recall Buzzard's construction of overconvergent automorphic forms following \cite[Section~5]{buzzard2}.

We fix a definite quaternion algebra $D$ over $\QQ$ which splits at $p$, and
we fix an isomorphism $D \otimes \QQ_p \simeq \rmM_2(\QQ_p)$ and identify them, so that the groups considered in  Subsection~\ref{SS:subgroups of GL2(Zp)} may be viewed as subgroups of $(D \otimes \QQ_p)^\times$.
We fix the tame level structure $K^p$ to be an open compact subgroup of $(D \otimes \AAA_f^{(p)})^\times$. We call $K^p$ \emph{neat} if it satisfies the following condition (see \cite[Section~4]{buzzard2}):
\[
(\textrm{Neat})\qquad 
\textrm{for any }x \in (D \otimes \AAA_f)^\times, \textrm{ we have }x^{-1} D^\times x \cap K^p \Iw_q = \{1\}.
\]
The neat condition is cofinal in the direct system of all tame level structures. For instance, for any $l \geq 3$, the full $l$ level structure is neat. 
 
Let $\chi$ be an $m_0$-locally analytic character of $\ZZ_p^\times$ with values in an affinoid $\QQ_p$-algebra $A$.  
For $m\geq m_0$, Buzzard \cite{buzzard2} defined the space of overconvergent modular forms of weight $\chi$ and radius of convergence $p^{-m}$ (with $m \in \NN$ and $m \geq 2$ if $p=2$) to be
\[
S^{D, \dagger, m}_{\chi}: = \big\{ 
\varphi: D^\times \backslash (D \otimes \AAA_f)^\times / K^p \to \Ind_{B(\ZZ_p)}^{\Iw_q}(\chi)^{m, \an}\; \big|\;\varphi(xu_p) = \varphi(x)|\!|^\chi_{u_p}, \textrm{ for }u_p \in \Iw_q
\big\}.
\]
We put $S^{D, \dagger}_\chi : = \cup_{m\geq m_0} S^{D, \dagger, m}_\chi$.
When $\chi = (k, \psi)$ is a classical character of conductor $p^{m_0}\leq p^m$, $S_\chi^{D, \dagger, m}$ contains the subspace of classical automorphic forms of weight $k+2$:
\[
S^D_{k+2}(K^p\Iw_{p^{m}}; \psi): = \big\{ 
\varphi: D^\times \backslash (D \otimes \AAA_f)^\times / K^p \to \Ind_{B(\ZZ_p)}^{\Iw_q}(\chi)^{m, \alg}\; \big|\;\varphi(xu_p) = \varphi(x)|\!|^\chi_{u_p}, \textrm{ for }u_p \in \Iw_q
\big\}.
\]

\subsection{The $U_p$-operator}
We choose a decomposition of the double coset:
\[
\Iw_q \big(\begin{smallmatrix}
p &0\\0&1
\end{smallmatrix} \big) \Iw_q = \coprod_{j=0}^{p-1} \Iw_q v_j,
\] for example with $v_j = \big(\begin{smallmatrix}
p &0\\jq&1
\end{smallmatrix} \big)$,
and define
\begin{equation}
\label{E:Up operator}
 U_p(\varphi) := \sum_{j=0}^{p-1} \varphi|^{\chi}_{v_j}, \quad \textrm{with }(\varphi|^{\chi}_{v_j})(g): = \varphi(gv_j^{-1})|\!|^{\chi}_{v_j}.
\end{equation}
This definition does not depend on the choice of the coset representatives $v_j$.
The operator $U_p$ naturally acts on $S_\chi^{D, \dagger, m}$, and preserves the subspace $S_{k+2}^D(K^p\Iw_{p^m}; \psi)$ if $\chi$ is a classical character $(k, \psi)$ of conductor $p^{m_0} \leq p^m$.

\begin{remark}
\label{R:Buzzard freedom}
Buzzard \cite{buzzard2} allows variants of the construction by taking equivariant functions under a (smaller) Iwahori subgroup $\Iw_{p^\alpha}$ and with values in $\Ind_{B(\ZZ_p)}^{\Iw_{p^\alpha}}(\chi)^{\beta, \an}$, namely, those functions that are analytic on $a+p^{\beta-\alpha}\ZZ_p$ for all $a \in \ZZ_p$ (under the identification $h(z) = f\big(\big( \begin{smallmatrix}
1&0 \\ p^{\alpha} z & 1
\end{smallmatrix}\big)\big)$), as long as $\alpha \geq 1$ and $\beta \geq \max\{\alpha, m_0, v(q)\}$. But \cite[Lemma~4]{buzzard2} says that the characteristic power series for the $U_p$-action on the space of overconvergent automorphic forms does not depend on this general flexibility.  
In 
\cite{wan-xiao-zhang}, we exploit the construction when $ \alpha =\beta = m_0$, i.e.,  taking equivariant functions under the Iwahori subgroup $\Iw_{p^{m_0}}$ and having the target of $\varphi$ to be $A\langle z\rangle$.
In contrast, we focus on the other extreme $\alpha = v(q)$ and $\beta = m_0$ in this paper: letting the disks become smaller and smaller as $m_0$ goes to infinity, while keeping the setup to take equivariant functions under the $\Iw_q$-action; this is adapted to the goal of finding an integral model.
\end{remark}

\subsection{Integral model of the space of $p$-adic automorphic forms}
\label{SS:naive integral model}
\quash{Following the key idea of Coleman \cite{coleman-halo}, 
we provide an integral model of the space of automorphic forms.}
Consider the universal character $[-]: \ZZ_p^\times \to\Lambda^\times$, viewed as a character of $B(\ZZ_p)$ by composing with $\pi_d$.  (So the coefficient ring will be the Iwasawa algebra $A = \Lambda$.) Define the space of \emph{integral $p$-adic automorphic forms for $D$} to be
\[
S^{D}_{\inte}: = \big\{ 
\varphi: D^\times \backslash (D \otimes \AAA_f)^\times / K^p \to \Ind_{B(\ZZ_p)}^{\Iw_q}([-])\; \big|\;\varphi(xu_p) = \varphi(x)|\!|^{[-]}_{u_p}, \textrm{ for }u_p \in \Iw_q
\big\}.
\]
Since the coefficient ring $\Lambda$ is not a Banach algebra, $S^{D}_\inte$ is not a Banach space in the literal sense. But as shown later in \eqref{E:explicit integral model} and Proposition~\ref{P:Up operator agrees}, at least when the tame level is neat, $S^{D}_\inte$ is the completion of an (countably) infinite direct sum of $\Lambda$'s.
Moreover, the topological space $S^{D}_\inte$ carries a continuous action of $U_p$, defined using the same formula~\eqref{E:Up operator}.  (Here and later, we do not discuss the tame Hecke operators as we will not use them. Clearly, this space carries a natural action of the tame Hecke algebra; see for example \cite[\S 3.4]{wan-xiao-zhang}.)

Note that $\Lambda$ (or $\Lambda[\frac 1p]$) is not an affinoid algebra, one needs to pass over an affinoid subdomain of $\calW$ to compare $S^{D}_\inte$ with the usual space of overconvergent automorphic forms.
More precisely, for $m_0 \in \NN_{\geq4}$, we write $\calW_{m_0} : = \calW^{\leq p^{-1/p^{m_0-4}}}$ and use $[-]_{m_0}: \ZZ_p^\times \to \calO_{\calW_{m_0}}^\times$ to denote the induced universal character (the radius here is not optimal), which is $m_0$-locally analytic. 
Then $S^{D}_\inte \widehat \otimes_\Lambda \calO_{\calW_{m_0}}$ contains $S^{D,\dagger}_{[-]_{m_0}}=\cup_{m\geq m_0}S^{D,\dagger,  m}_{[-]_{m_0}}$ and hence $S^{D,\dagger,  m_0}_{[-]_{m_0}}$ as subspaces.

\begin{remark}\label{R:Caveat}
In his thesis \cite{gouvea}, F. Gouv\^ea proved that $U_p$ is not a compact operator on the space of $p$-adic modular forms. Since $S^{D}_\inte$ is an integral model of the space of $p$-adic automorphic forms for $D$, it is very likely that the $U_p$-action on $S^{D}_\inte$ is not compact in the sense of Definition~\ref{D:compact operator}. However, a key observation of this work is that one can still define the characteristic power series for the $U_p$-action with respect to a carefully chosen orthonormal basis of $S^{D}_\inte$. See Theorem \ref{T:HP lower bound} and the discussion in Section~\ref{Sec:integral model}.
\end{remark}

\begin{remark}
\label{R:relation to Emerton}
The $\Lambda$-dual space of $S^{D}_\inte$ may be identified with $P(\ZZ_p) = \big(\begin{smallmatrix} \ZZ_p^\times & \ZZ_p \\ 0& 1\end{smallmatrix} \big)$-coinvariants of Emerton's completed homology of the Shimura variety associated to $D^\times$:
\[
\tilde S^D: = \Big( \varprojlim_{m \to \infty} \ZZ_p\big[D^\times \backslash (D \otimes \AAA_f)^\times) / K^pK_{p,m} \big] \Big)_{P(\ZZ_p)},
\]
where $K_{p,m} : = \big(1+ p^m \rmM_2(\ZZ_p) \big)^\times$. (Rigorously speaking, the $U_p$-action on both sides might be normalized slightly differently; this is because we chose to follow the convention of Buzzard \cite{buzzard2}.)
More naturally, we should have taken the invariants under $N(\ZZ_p)$, which would lead to a theory of eigensurface over $\calW \times \calW$ ``homogeneous" along one factor of $\calW$. But it is custom to simplify the picture by taking a ``slice" of the eigensurface to study the eigencurve.

For general algebraic groups $G$ over $\QQ$ which is quasi-split at $p$ and compact modulo center at the archimedean place, one can construct the corresponding integral model by taking the coinvariants of the completed homology of the Shimura variety under the unipotent subgroup of the chosen Borel subgroup at $p$, or equivalently consider a similar induced representation. These two viewpoints are essentially the same, as explained in \cite[Section~3.10]{loeffler}.
\quash{We expect that Caveat~\ref{Caveat} still presents a problem, which we hope can be handled using a construction similar to that in Section~\ref{Sec:integral model}.}
More generally, it might be possible to extend this construction to general $G$ by looking at the completed homology groups of the associated locally symmetric space.

We also point out that our construction is closely tied to the \'etale (or Betti) realization of the eigencurve or the corresponding automorphic forms.
One can also realize the space of automorphic forms using their de Rham realization, as in the original construction of Coleman--Mazur \cite{coleman-mazur} and Andreatta--Iovita--Pilloni--Stevens \cite{AIS,pilloni}.
 In the recent work of Andreatta, Iovita, and Pilloni \cite{AIP}, they constructed some integral model of the space of overconvergent Hilbert modular forms, developing an idea of Coleman \cite{coleman-halo} (see also Remark \ref{R:final}). \end{remark}

\begin{hypothesis}
\label{H:neat}
For simplicity, from now on we assume that $K^p$ is neat. But we shall insert discussions throughout on reducing the argument from the general case to the neat case.
\end{hypothesis}

\subsection{Explicit presentation of (overconvergent) automorphic forms}
\label{SS:explicit direct sum presentation} One can give an explicit description of the space of (overconvergent) automorphic forms.
For this,  we decompose $(D \otimes \AAA_f)^\times$ into (a disjoint union of) double cosets $\coprod_{i=0}^{t-1} D^\times \gamma_i K^p \Iw_q$, for some elements $\gamma_0,\gamma_1, \dots, \gamma_{t-1} \in (D \otimes \AAA_f)^\times$.
 The condition (Neat) implies that the natural map $D^\times \times K^p \Iw_q \to D^\times \gamma_i K^p\Iw_q$ for each $i$ sending $(\delta,u)$ to $\delta\gamma_iu$ is bijective. 
Since $D^\times$ is dense in $(D \otimes \QQ_p)^\times$, we may and will take each $\gamma_i$ so that its $p$-component $\gamma_{i,p}$ is just $1$.\footnote{It was pointed out to us by  Emerton and  Ren (independently) that the assumption $\gamma_{i,p}=1$ is not essentially needed to prove Proposition~\ref{P:explicit Up}(3) later.}

Evaluating each function $\varphi$ at these chosen $\gamma_i$'s, we have an explicit description of various spaces of automorphic forms:
\begin{equation}
\label{E:explicit integral model}
\xymatrix@R=0pt{
S^{D}_\inte \ar[r]^-\cong & \bigoplus_{i=0}^{t-1} \calC(\ZZ_p; \Lambda) 
\\ S^{D, \dagger, m}_{\chi} \ar[r]^-\cong & \bigoplus_{i=0}^{t-1} \calO_{\bfB_{qp^{-m}}} \widehat \otimes_{\QQ_p} A
\\
\varphi \ar@{|->}[r] & \big( \varphi(\gamma_i) \big)_{i =0, \dots, t-1},
}
\end{equation}
where $\chi: \ZZ_p^\times  \to A^\times$ is an $m$-analytic character of $\ZZ_p^\times$ with values in an affinoid $\QQ_p$-algebra $A$.

\subsection{Characteristic power series}
\label{SS:char power series}
From the explicit presentation \eqref{E:explicit integral model}, we see that $S_\chi^{D, \dagger, m}$ is an orthonormalizable Banach $A$-module, i.e. there exists an orthonormal basis $(e_i)_{i\in \ZZ_{\geq 0}}$ such that $S_\chi^{D, \dagger, m} \cong \widehat \oplus_{i \geq 0} A e_i$.
With respect to this basis, the action of $U_p$ is given by an infinite matrix, say $P$.
Moreover, the action of $U_p$ is compact (see e.g. \cite[Lemma~12.2]{buzzard}), namely, it is a uniform limit of $A$-linear operators whose images are finite $A$-modules.
We define the \emph{characteristic power series} of the $U_p$-action on $S_\chi^{D, \dagger, m}$ to be
\[
\Char(U_p; S_\chi^{D, \dagger, m}): = \det (I_\infty - XP) \in A\llbracket X\rrbracket.
\]
This power series converges and does not depend on the choice of the orthonormal basis $(e_i)_{i \in \ZZ_{\geq 0}}$.

\begin{definition}
For $r \in (0,1) \cap p^\QQ$, let $[-]_{\leq r}: \ZZ_p^\times \to \calO_{\calW^{\leq r}}^\times$ denote the universal character.
Choose $m \in \NN$ such that $r < p^{-q/p^{m}(p-1)}$ so that the universal character is $m$-locally analytic.
The \emph{spectral curve} $\Spc_D^{\leq r}$ over $\calW^{\leq r}$ is defined to be the (scheme theoretic) zero locus of the characteristic power series $\Char(U_p; S_{[-]_{\leq r}}^{D, \dagger, m})$ inside $\calW^{\leq r} \times \GG_m^\rig$.
By \cite[Lemma~4]{buzzard2}, $\Char(U_p; S_{[-]_{\leq r}}^{D, \dagger, m})$ and hence $\Spc_D^{\leq r}$ does not depend on the choice of $m$, and it is compatible as $r$ varies.
We put $\Spc_D: = \cup_{r \in (0,1)} \Spc_D^{\leq r}$ to be the \emph{spectral curve} with tame level $K^p$.
The natural projection wt to weight space $\calW$ is called the \emph{weight map}.  The composition of $x \mapsto x^{-1}$ on $\GG_m^\rig$ and the natural projection $\Spc_D \to \GG_m^\rig$ is called the \emph{slope map} and denoted by $a_p: \Spc_D \to \GG_m^\rig$.
\end{definition}

\begin{remark}\label{R:non-neat}
One may reduce the non-neat case to the neat case as follows. Choose a split prime $l \geq 3$ of $D$ different from $p$, and intersect $K^p$ with the full level $l$ structure. The spectral curve for this new tame level structure (which is neat) is endowed with a $\GL_2(\ZZ/l\ZZ)$-action which commutes with the weight map. Since $|\GL_2(\ZZ/l\ZZ)|$ is invertible on weight space, the spectral curve for $K^p$ is a union of connected components of the spectral curve for this new tame level. 
\end{remark}

We record the following classicality result for later use.  It establishes a basic link between the classical theory of automorphic forms and the theory of the overconvergent ones.
\begin{proposition}[{\cite[Proposition~4]{buzzard2}}]
\label{P:classicality}
Let $\chi = (k, \psi): \ZZ_p^\times \to E^\times$ be a classical character of conductor $p^m$.
Let $0 \neq \varphi \in S^{D, \dagger}_\chi$ be an eigenvector for $U_p$ with non-zero eigenvalue $\lambda$.
\begin{itemize}
\item
If $v(\lambda)<k+1$, then $\varphi$ is classical, i.e. $\varphi \in S_{k+2}^D(K^p\Iw_{p^m}; \psi)$.
\item
If $v(\lambda)>k+1$, then $\varphi \notin S_{k+2}^D(K^p\Iw_{p^m}; \psi)$.
\end{itemize}

\end{proposition}

\subsection{One-variable $p$-adic analysis}
\label{SS:Mahler expansion I}
Before proceeding, we need some one-variable $p$-adic analysis, as developed by P. Colmez \cite{colmez}.

Recall that $\calC(\ZZ_p;\ZZ_p)$ carries a supremum norm: 
\[
\textrm{ for }f \in\calC(\ZZ_p; \ZZ_p), \ |f|_{\ZZ_p}: =\max_{z \in \ZZ_p} |f(z)|.
\]
The functions $1, z, \dots, \binom zn, \dots$ form an orthonormal basis of $\calC(\ZZ_p;\ZZ_p)$, called the \emph{Mahler basis}. In other words, any $f \in \calC(\ZZ_p;\ZZ_p)$ admits a \emph{Mahler expansion}:
\begin{equation}
\label{E:f(z)}
f(z) = \sum_{n \geq 0} a_n \binom zn, \textrm{ where all }a_n \in \ZZ_p;\textrm{ and }\lim_{n \to \infty} |a_n| =0.
\end{equation}

For $m \geq v(q)$, recall that $\calO_{\bfB_{qp^{-m}}}$ is the  subspace of $\calC(\ZZ_p; \QQ_p)$ consisting of those continuous functions that admits a local Taylor expansion over $a+q^{-1}p^m\ZZ_p$ for each $a \in \ZZ_p$.
It carries a natural norm
\[
\textrm{ for }f \in\calO_{\bfB_{qp^{-m}}}, \ |f|_{{qp^{-m}}, \an}: =\max_{z \in \bfB_{qp^{-m}}(\CC_p)} |f(z)|.
\]
By \cite[Th\'eor\`em~1.29]{colmez}, the functions
\[
\Big\lfloor \frac n {q^{-1}p^m} \Big\rfloor !\cdot  \binom zn; \quad n \in \ZZ_{\geq 0}
\]
form an orthonormal basis of $\calO_{\bfB_{qp^{-m}}}$ for the norm $|\cdot|_{{qp^{-m}}, \an}$.
Note that this orthonormal basis is not the one we commonly use in the
context of studying $S^{D, \dagger,m}_{[-]_m}$ (e.g. in \cite{buzzard2} and
in \cite{wan-xiao-zhang}).

In the rest of the paper, we always equip $\Lambda\llbracket X\rrbracket$ with the $\gothm_\Lambda$-adic topology.

\begin{proposition}
\label{P:Up operator agrees}
Using the isomorphism \eqref{E:explicit integral model}, the space $S^{D}_\inte$ admits an orthonormal basis (over $\Lambda$) given by 
\[
1_0, \dots, 1_{t-1}, z_0, \dots, z_{t-1}, \tbinom {z_0}2, \dots, \tbinom{z_i}{2}, \dots,
\]
where the subscript $i$ indicates that the term comes from the $i$th direct summand.
Let $P = (P_{m,n})_{m,n \in \ZZ_{\geq 0}}$ denote the corresponding infinite matrix for the $U_p$-action, with coefficients in $\Lambda$.
Suppose that the (uniform) limit of power series
\[
\Char(P): = \det(I_\infty - XP) = \lim_{n \to \infty} \det \Big(I_n - X (P_{i,j}\big)_{i,j =0, \dots, n-1}\Big) \in \Lambda \llbracket X\rrbracket
\]
exists (which we shall prove in Theorem~\ref{T:HP lower bound}).
Then it agrees with the characteristic power series of the $U_p$-action on each $S^{D, \dagger, m}_{[-]_m}$.
\end{proposition}
\begin{proof}
By \cite[Th\'eor\`em~1.29]{colmez} we cited above, the functions
\[
\Big\lfloor\frac n{q^{-1}p^m}\Big\rfloor!\cdot \binom{z_i}n \textrm{ for }i =0, \dots, t-1 \textrm{ and }n \in \ZZ_{\geq 0}
\]
form an orthonormal basis of $S^{D, \dagger, m}_{[-]_m}$.
If $P'$ denotes the infinite matrix of $U_p$-action on this basis, then
$P$ and $P'$ are conjugated by an infinite diagonal matrix with diagonal entries
\[
\underbrace{1, \dots, 1}_t, \underbrace{1, \dots, 1}_t, \dots, \underbrace{\lfloor \tfrac n{q^{-1}p^m}\rfloor !, \dots, \lfloor \tfrac n{q^{-1}p^m}\rfloor !}_{t}, \dots
\]
So taking the limit of the characteristic polynomial of the first $r\times r$-minors, as $r$ goes to infinity gives
\[
\Char(U_p; S^{D, \dagger, m}_{[-]_m}) = \det(I_\infty - XP') = \det(I_\infty- XP),
\]
provided that the latter is well defined.
\end{proof}

\begin{remark}
Once again, we point out that the definition of $\Char (P)$ depends on the choice of the orthonormal basis; in particular, it a priori depends on the choice of the coset representatives (see Remark~\ref{R:Up almost compact}), if we had not shown that it agrees with the characteristic power series $\Char(U_p; S_{[-]_m}^{D, \dagger, m})$. See Section~\ref{Sec:integral model} for more discussion.
\end{remark}

\section{Estimation of the Newton polygon}
\label{Sec:Estimation of NP}

The advantage of working with a definite quaternion algebra is that the action of the $U_p$-operators may be written in a relatively explicit form.
This was first observed by Buzzard and carried out by his student Jacobs \cite{jacobs} (in one example), and later carefully optimized by the second and third authors and Zhang in \cite{wan-xiao-zhang}.

In this section, we will first revisit this explicit presentation of the $U_p$-operator.  Then we give an estimate of the explicit formula for the $U_p$-operator and provide a lower bound of the Newton polygon for the $U_p$-action that is valid over weight space $\calW^{>1/p}$. Luckily, this lower bound agrees with the actual Newton polygon at infinitely many points.
Theorem~\ref{T:main theorem} follows from a careful analysis of the Newton polygon at these points, as proved at the end of this section.

\begin{proposition}
\label{P:explicit Up}
In terms of the isomorphism~\eqref{E:explicit integral model}, the $U_p$-operator on $S^{D}_\inte$ can be described by the following commutative diagram.
\[
\xymatrix@C=50pt{
S^{D}_\inte\ar[r]^-{\varphi \mapsto (\varphi(\gamma_i))} \ar[d]_{\varphi \mapsto U_p\varphi} &
\oplus_{i=0}^{t-1} \calC(\ZZ_p; \Lambda) \ar[d]^{\gothU_p}
\\
S^{D}_\inte\ar[r]^-{\varphi \mapsto (\varphi(\gamma_i))} &
\oplus_{i=0}^{t-1} \calC(\ZZ_p; \Lambda).
}
\]
Here the right vertical arrow $\gothU_p$ is given by a $t\times t$ matrix with the following description.
\begin{itemize}
\item[(1)] Each entry of $\gothU_p$ is a sum of operators of the form $|\!|^{[-]}_{\delta_p}$, where $\delta_p$ is the $p$-component of a \emph{global element $\delta \in D^\times$}.
\item[(2)] There are exactly $p$ such operators appearing in each row and each column  of $\gothU_p$.
\item[(3)] Each $\delta_p$ appearing above belongs to $
 \big(\begin{smallmatrix} p\ZZ_p& \ZZ_p\\q\ZZ_p&\ZZ_p^\times \end{smallmatrix}\big)
$.
\end{itemize}
\end{proposition}
\begin{proof}
We reproduce the proof from \cite[Proposition~4.4]{wan-xiao-zhang} to make this paper more self-contained.
For each $\gamma_i$, we have
\[
(U_p \varphi)(\gamma_i) = \sum_{j=0}^{p-1}
\varphi(\gamma_i v_j^{-1})|\!|^{[-]}_{v_j}.
\]
Write each $\gamma_iv_j^{-1}$ \emph{uniquely} as $\delta_{i,j}^{-1} \gamma_{\lambda_{i,j}} u_{i,j}$ with $\delta_{i,j} \in D^\times$, $\lambda_{i,j} \in \{0, \dots, t-1\}$, and $u_{i,j} \in K^p\Iw_q$.
Then we have
\[
(U_p \varphi)(\gamma_i) = \sum_{j=0}^{p-1}
\varphi(\delta_{i,j}^{-1} \gamma_{\lambda_{i,j}} u_{i,j})|\!|^{[-]}_{v_j} = \sum_{j=0}^{p-1}
\varphi( \gamma_{\lambda_{i,j}})|\!|^{[-]}_{u_{i,j,p}v_j},
\]
where $u_{i,j,p}$ is the $p$-component of $u_{i,j}$.
Substitute back in $u_{i,j}v_j = \gamma_{\lambda_{i,j}}^{-1} \delta_{i,j} \gamma_i$ and note the fact that both $\gamma_i$ and $\gamma_{\lambda_{i,j}}$ have trivial $p$-component by our choice in Subsection~\ref{SS:explicit direct sum presentation}.  We have
\[
(U_p\varphi)(\gamma_i) =
\sum_{j=0}^{p-1} \varphi(\gamma_{\lambda_{i,j}}) |\!|^{[-]}_{\delta_{i,j,p}},
\]
where $\delta_{i,j,p}$ is the $p$-component of the \emph{global element} $\delta_{i,j} \in D^\times$.
We now check the description of each $\delta_{i,j,p}$:
\[
\delta_{i,j,p} = u_{i,j,p}v_j \in \Iw_q\big(\begin{smallmatrix} p&0\\0&1 \end{smallmatrix}\big)\Iw_q \subseteq\big(\begin{smallmatrix} p\ZZ_p& \ZZ_p\\q\ZZ_p&\ZZ_p^\times \end{smallmatrix}\big).
\]
 This concludes the proof of the proposition.
\end{proof}

\begin{remark}
By choosing the representatives $\gamma_i$'s more carefully, one can ensure that each global element $\delta$ appearing above has norm exactly $p$. This was used in a somewhat crucial way in \cite{wan-xiao-zhang}.
\end{remark}

Now, to understand the action of the $U_p$-operator, it suffices to understand the action of $|\!|^{[-]}_{\delta_p}$ on $\calC(\ZZ_p; \Lambda)$ for each $\delta_p = \big( \begin{smallmatrix}
a&b \\c&d
\end{smallmatrix} \big) \in \big(\begin{smallmatrix} p\ZZ_p& \ZZ_p\\q\ZZ_p&\ZZ_p^\times \end{smallmatrix}\big)$.
For later use, we will generalize our discussion to all $\delta_p $ in the monoid $\bfM_1$ (as defined in \eqref{E:monoid M1}).

\subsection{More Mahler expansions}
\label{SS:more Mahler expansions}
Recall that every function $f \in \calC(\ZZ_p; \ZZ_p)$ admits a Mahler expansion $f(z) = \sum_{n \geq 0} a_n(f)\binom zn$ with $a_n(f) \in \ZZ_p$ and $\lim_{n \to \infty} a_n(f) =0$.
These \emph{Mahler coefficients} $a_n(f)$ can be determined by the following process:
for $f(z) \in  \calC(\ZZ_p; \ZZ_p)$, we write 
\[
\tilde \Delta(f)(z) = \tilde \Delta^{(1)}(f)(z)
 = f(z+1) -f(z) \textrm{ and } \tilde \Delta^{(m+1)}(f)(z) = \tilde \Delta \big(\tilde \Delta^{(m)}(f)\big)(z) \textrm{ for }m \in \NN.
\]
Set $\tilde \Delta^{(0)}(f) = f$.
Then for $m \in \ZZ_{\geq 0}$, we have
\begin{equation}
\label{E: expression for an(f)}
a_m(f) =  \tilde \Delta^{(m)}(f)(0) = \sum_{i=0}^m (-1)^{m-i}\binom mi f(i).
\end{equation}

\begin{proposition}
Consider the action of $\delta_p = \big( \begin{smallmatrix}
a&b \\c&d
\end{smallmatrix} \big) \in \bfM_1$
on $\calC(\ZZ_p; \Lambda)$.
Let $P(\delta_p)=(P_{m,n}(\delta_p))_{m,n\in \ZZ_{\geq 0}}$ denote the infinite matrix for this action with respect to the orthonormal Mahler basis $1, z, \binom z2, \dots$.
Then
\begin{equation}
\label{E:Pmn}
P_{m,n}(\delta_p) = \tilde  \Delta^{(m)}\bigg( \binom{(az+b)/(cz+d)}{n} \cdot \big[(cz+d)\big]\bigg)\bigg|_{z=0}.
\end{equation}
\end{proposition}
\begin{proof}
For a function $f \in \calC(\ZZ_p; \Lambda)$, we write its Mahler coefficients $a_n(f) \in \Lambda$ as an infinite column vector.
So the entry $P_{m,n}(\delta_p)$ corresponds to the $m$th Mahler coefficient of the function
\[
\binom zn\Big|\!\Big|^{[-]}_{ \big( \begin{smallmatrix}
a&b \\c&d
\end{smallmatrix} \big)}  = \binom{(az+b)/(cz+d)}{n} \cdot \big[(cz+d)\big].
\] 
Note that the universal character $[-]$ cannot be applied to the formal expression $cz+d$; but $[cz+d]$ still makes sense as a continuous function on $\ZZ_p$ with values in $\Lambda$.
Comparing with \eqref{E: expression for an(f)}, we have
\[
P_{m,n}(\delta_p) = a_m \bigg( \binom{(az+b)/(cz+d)}{n} \cdot \big[(cz+d)\big]\bigg) = \tilde  \Delta^{(m)}\bigg( \binom{(az+b)/(cz+d)}{n} \cdot \big[(cz+d)\big]\bigg)\bigg|_{z=0}.
\]
\end{proof}

\subsection{More $p$-adic analysis}
We start by listing the following three useful equalities, which can be checked easily.

\begin{align}
\label{E:Delt of product}
\tilde \Delta^{(m)}(fg)(z) & = \sum_{i=0}^m  \binom mi \tilde \Delta^{(m-i)} (f)(z+i) \tilde \Delta^{(i)}(g)(z).
\\
\label{E:Delta of binom zm}
\tilde \Delta^{(m)}\binom zn & = \begin{cases}
\binom z{n-m} & \textrm{if }n \geq m\\
0 & \textrm{otherwise}.
\end{cases}
\\
\label{E:binomial identity}
\binom{x_1+\cdots + x_r}n &= \sum_{i_1+\cdots + i_r = n}\binom{x_1}{i_1} \cdots\binom{x_r}{i_r}.
\end{align}
As a corollary of \eqref{E:Delta of binom zm}, if $f = \sum_{n \geq 0} a_n(f) \binom zn$ is the Mahler expansion of a function $f \in \calC(\ZZ_p; \ZZ_p)$, then
\begin{equation}
\label{E:Delta m of f}
\tilde \Delta^{(m)}(f)(z) = \sum_{n \geq m} a_n(f) \binom z{n-m}.
\end{equation}

\begin{definition}\label{D:degree}
We say a continuous function $f \in \calC(\ZZ_p; \ZZ_p)$ is a \emph{polynomial function of degree $\leq n$} (for $n \in \ZZ_{\geq 0}$) if the Mahler coefficients $a_i(f) = 0$ for $i >n$.\end{definition}

\begin{lemma}
\label{L:degree basic properties}

\begin{enumerate}
\item
A continuous function $f \in \calC(\ZZ_p; \ZZ_p)$ is a polynomial function of degree $\leq n$ if and only if $\tilde \Delta^{(n+1)}(f) = 0$.
\item
If $f$ and $g$ are polynomial functions on $\ZZ_p$ of degree $\leq n$ and $\leq m$, respectively, then $fg$ is a polynomial function of degree $\leq m+n$.

\item
If $f \in \calC(\ZZ_p; \ZZ_p)$ is a polynomial function of degree $\leq r$, then $\binom {f(z)}n$ is a polynomial function of degree $\leq rn$.
\end{enumerate}
\end{lemma}
\begin{proof}

(1) If $f(z) = \sum_{i \geq 0} a_i(f)\binom zi$ is the Mahler expansion of $f$, then 
$$\tilde \Delta^{(n+1)}(f) = \sum_{i \geq n+1} a_i(f) \binom z{i-n-1}.$$
Then $\tilde \Delta^{n+1}(f) = 0$ if and only if $a_{i} =0$ for all $i \geq n+1$.  (1) is proved.

The rest of the lemma follows immediately from the fact that the polynomial functions of degree $\leq n$ are precisely the polynomials in $\QQ_p[z]$ of degree at most $n$ which map $\ZZ_p$ to $\ZZ_p$.

\quash{(2) By \eqref{E:Delt of product}, we have
\[
\tilde \Delta^{(m+n+1)}(fg)(z)  = \sum_{i=0}^{m+n+1} \binom {m+n+1}i \tilde \Delta^{(m+n+1-i)} (f)(z+i) \tilde \Delta^{(i)}(g)(z).
\]
For $i =0, \dots, m$, $\tilde \Delta^{(m+n+1-i)}(f)(z+i) = 0$; and for $i = m+1, \dots, m+n+1$, $\tilde \Delta^{(i)} (g)(z) =0$. So we conclude that $\tilde \Delta^{(m+n+1)}(fg)(z)  =0$ and $fg$ has degree $\leq m+n$ by (1).

(3) By (2) above and \eqref{E:binomial identity},
it suffices to prove this for $f(z) = a\binom zm$ with $a \in \ZZ_p$. We use induction on $m+n$, where the case when $m$ or $n$ is $1$ is clear, serving as the inductive base. So we may assume that $m, n \geq 2$.
Since it is clear that $\binom{f(z)}n$ is a continuous $\ZZ_p$-valued function on $\ZZ_p$, by (1), it is enough to check that $\tilde \Delta^{(mn+1)}(\binom{f(z)}n) =0$. Indeed,
\begin{align*}
\tilde \Delta \Big(\binom{a\binom zm}n \Big) &  = \binom{a\binom {z+1}m}n-\binom{a\binom zm}n 
 = \binom{a\binom zm +a\binom z{m-1}}n-\binom{a\binom zm}n  
 \\
 &\stackrel{\eqref{E:binomial identity}}= \sum_{i=1}^n \binom{a\binom zm}{n-i} \binom{a\binom z{m-1}}i.
\end{align*}
By inductive hypothesis and (2), 
the $i$th term of the sum has degree $\leq m(n-i) + (m-1)i = mn-i$; so by (1), it is zero after applying the operation $\tilde \Delta^{(mn)}$.  This concludes the proof of (3).}
\end{proof}

In addition to the degree, the following convenient definition is tailored for our computation.

\begin{definition-proposition}
\label{DP:tilted degree}
We say a continuous function $f \in \calC(\ZZ_p; \ZZ_p)$ has \emph{tilted degree $\leq n$} (for $n \in \ZZ_{\geq0}$) if the following equivalent conditions are satisfied:
\begin{enumerate}
\item
for any $m \in \NN$, $\tilde \Delta^{(m)}(f)$ is a (continuous) function on $\ZZ_p$ that takes value in $p^{m-n}\ZZ_p$;
\item
writing down the Mahler expansion of $f(z) = \sum_{j \geq 0} a_j(f)\binom zj$, then $v(a_j(f)) \geq j-n$.
\end{enumerate}
Note that the assumption $f \in \calC(\ZZ_p; \ZZ_p)$ implies that condition (1) for $m \leq n$ and condition (2) for $j \leq n$ hold automatically.
\end{definition-proposition}

\begin{proof}
Let $f(z) = \sum_{j \geq 0} a_j(f)\binom zj$ be the Mahler expansion of $f$.
By \eqref{E:Delta m of f}, $\tilde \Delta^{(m)}(f)  = \sum_{n \geq m} a_n(f) \binom z{n-m}$.
Since the Mahler basis forms an orthonormal basis of $\calC(\ZZ_p; \ZZ_p)$, $\tilde\Delta^{(m)}(f)$ takes value in $p^{m-n}\ZZ_p$ if and only if $a_j(f) \in p^{m-n}\ZZ_p$ for all $j \geq m \geq n$, which is equivalent to $v(a_j(f)) \geq j-n$ for all $j \geq n$.
\end{proof}

\begin{remark}\label{R:degree}
If $f \in \calC(\ZZ_p; \ZZ_p)$ is a polynomial function of degree $\leq n$, then it has tilted degree $\leq n$.
\end{remark}

\begin{lemma}
\label{L:tilted degree additive}
\begin{enumerate}
\item
If $f \in \calC(\ZZ_p; \ZZ_p)$ and $n \in \NN$, then $f$ has tilted degree $\leq n$ if and only if $\tilde \Delta(f)$ has tilted degree $\leq n-1$.
\item
If $f$ and $g$ are $\ZZ_p$-valued continuous functions on $\ZZ_p$ of tilted degree $\leq m$ and $\leq n$, respectively.
Then $fg$ has tilted degree $\leq m+n$.
\end{enumerate}
\end{lemma}
\begin{proof}
(1) is clear from Definition-Proposition~\ref{DP:tilted degree}(2) because $a_j(f) = a_{j-1}(\tilde \Delta(f))$.

(2)
We check it using Definition-Proposition~\ref{DP:tilted degree}(1).
By \eqref{E:Delt of product}, we have
\[
\tilde \Delta^{(r)}(fg)(z)  = \sum_{i=0}^{r} \binom {r}i \tilde \Delta^{(r-i)} (f)(z+i) \tilde \Delta^{(i)}(g)(z).
\]
Each term on the right hand side has valuation at least $(r-i-m) + (i-n) = r-m-n$.
So $fg$ has tilted degree $\leq m+n$.
\end{proof}

To understand the expression \eqref{E:Pmn}, we need the following estimates.

\begin{lemma}\label{L:tilted degree estimate 0}
Let $g(z) = b_0 + b_1z + \cdots + b_r z^r\in\ZZ_p[z]$. If  there exists some $s\in\RR_{\geq0}$ such that $v(b_i)\geq is$ for all $i$, then we can rewrite $g(z)$ as
\[
\sum_{i=0}^r c_i\cdot i!\binom zi
\]
with $v(c_i)\geq is$ for all $i$.
\end{lemma}
\begin{proof}: If we compare the coefficients of each degree, we see that
\begin{align*}
b_r& = c_r,\\
b_{r-1}  & = c_{r-1}+ c_r  \alpha_{r,r-1},\\
b_{r-2} & = c_{r-2} + c_{r-1}  \alpha_{r-1, r-2} + c_r \alpha_{r, r-2},\\
\dots & \dots,
\end{align*}
where $\alpha_{i,j}$ are the coefficients of $z^j$ in the product $z(z-1)\cdots (z-i+1)$, which is of course an integer.
By reverse induction, we see that $v(b_r)\geq rs$ implies that of $v(c_r)\geq rs$, and $v(b_{r-1})\geq (r-1)s$ implies that of $c_{r-1}$, .... This concludes the lemma.
\end{proof}

\begin{lemma}
\label{L:tilted degree estimate I}
For a function $f(z) = a_0 + pa_1 z+ p^2a_2z^2 + \cdots \in \ZZ_p\llbracket pz \rrbracket$ and an integer $n \geq0$, the expression $\binom{f(z)}n$ has tilted degree $\leq \lfloor \frac np \rfloor$.
\end{lemma}
\begin{proof}
By approximation, we may assume that $f(z)$ is a polynomial.
Put 
\[
g(z) = f(z) (f(z)-1) \cdots (f(z)-n+1) \in \ZZ_p[pz]
\]
so that $\binom {f(z)}n = \frac 1{n!}g(z)$. By Lemma \ref{L:tilted degree estimate 0}, we may rewrite $g(z)$ as
\[
\sum_{k=0}^r c_k\cdot k!p^k\binom zk
\]
with $v(c_k)\in\ZZ_p$ for all $k$.

Since $\binom{f(z)}n$ is a continuous $\ZZ_p$-valued function, it suffices to prove that when $k > \lfloor \frac np\rfloor$, $v(k!p^k/n!) \geq k-\lfloor \frac np\rfloor$.
To this end, note that in this case $\lfloor \frac k{p^\ell}\rfloor \geq \lfloor \frac n{p^{\ell+1}}\rfloor$ for any $\ell \in \NN$.
Thus 
\[
v(k!p^k) = k + \lfloor \tfrac kp \rfloor + \lfloor \tfrac k{p^2} \rfloor + \cdots \geq k + \lfloor \tfrac n{p^2} \rfloor + \lfloor \tfrac n{p^3} \rfloor + \cdots = k-\lfloor \tfrac np \rfloor + v(n!).
\]
We are done.
\end{proof}

\begin{lemma}
\label{L:tilted degree estimate II}
For a function $f(z) = a_0+a_1z + pa_2z^2/2 + p^2a_3z^3 / 3 + \cdots + p^{k-1}a_k z^k / k + \cdots $ with $a_n \in \ZZ_p$ and an integer $m \geq 0$, the expression $\binom{f(z)}m$ has tilted degree $\leq m$.
\end{lemma}
\begin{proof}
By approximation, we may suppose that $f(z)$ is a polynomial function.
Moreover, by the binomial identity \eqref{E:binomial identity} together with the additive property of the tilted degree (Lemma~\ref{L:tilted degree additive}), we may assume that $f(z) = p^{n-1}a_nz^n/n$ is a monomial with $n \in \NN$ (the case $f(z)=a_0$ is trivial).
When $n=1$, this follows from the easy bound: $\binom{a_1z}m$ is a polynomial function of degree $m$, so we are done by Remark \ref{R:degree}.
So we assume $n>1$ for the rest of the proof.

Fix $n$ and put $g(z) = f(z)(f(z)-1) \cdots (f(z) +m-1) \in \ZZ_p\big[ \frac{p^{n-1}}{n}z^n\big]$; its degree is $nm$. By Lemma \ref{L:tilted degree estimate 0}, we can rewrite $\binom{f(z)}m = g(z) / m!$ as
\[
\sum_{k=0}^{nm} c_k k! p^{\frac{n-1-v(n)}{n}k} / m! \cdot  \binom zk,
\]
where $c_k$ belongs to $\ZZ_p[p^{1/n}]$.
Now we hope to show that for any $k\leq mn$, 
\begin{equation}
\label{E:inequality II}
v\big (c_k k! p^{\frac{n-1-v(n)}n k} / m!\big) \geq \max\{k-m, 0\}.
\end{equation}
When $k \leq m$, this follows from the fact that $\binom{f(z)}m$ is a continuous $\ZZ_p$-valued function on $\ZZ_p$.
So we may assume that $k>m$ (and $k \leq mn$ as the degree of $g(z)$ is just $mn$).
Simplifying the terms of \eqref{E:inequality II}, we see that it suffices to show (by ignoring the contribution of $v(c_k)$)
\begin{equation}
\label{E:stronger inequality}
v(k!) + m \geq v(m!) + \frac{1+v(n)}n{k}.
\end{equation}
We separate several cases:
\begin{itemize}
\item[(a)] If $p \nmid n$, then we need to show that $v(k!) + m \geq v(m!) + k/n$.  But this is just a combination of $v(k!) \geq v(m!)$ (as $k \geq m$) and $m \geq k/n$ (as $k \leq nm$).
\item[(b)]
If $n \geq 2p$ and $(p, n)\neq(2, 4)$, then $\frac{1+v(n)}{n} \leq 1/p$.
But $v(k!/m!) \geq \lfloor\frac {k-m}p \rfloor\geq \frac{k-m}p-\frac {p-1}p$.
So \eqref{E:stronger inequality} follows from
\[
v(k!/m!) + m \geq \tfrac kp + \tfrac{p-1}pm - \tfrac{p-1}p \geq \tfrac kp \geq \tfrac{1+v(n)}n k.
\]
\item[(c)]
If $n =p$, the inequality \eqref{E:stronger inequality} might fail.  So we have to go back to the beginning to show directly that $\binom{ap^{p-2}z^p}m$ has tilted degree $\leq m$, when $a \in \ZZ_p$.
We prove this by induction on $m$; the case of $m=0$ is void.
Now suppose this is proved for all numbers strictly less than $m$ and we prove it for $m$.
We need to show that
\[
\tilde \Delta\bigg( \binom{ap^{p-2}z^p}m\bigg) \textrm{ has tilted degree }\leq m-1.
\]
We argue as in Lemma~\ref{L:degree basic properties}(3). Note that
\begin{align*}
\tilde \Delta\big( \binom{ap^{p-2}z^p}m\big) &=  \binom{ap^{p-2}(z+1)^p}m - \binom{ap^{p-2}z^p}m
\\
&\stackrel{\eqref{E:binomial identity}}= \sum_{j=1}^m 
\binom{ap^{p-2}((z+1)^p-z^p)}j
 \binom{ap^{p-2}z^p}{m-j}.
\end{align*}
Note that $ap^{p-2}((z+1)^p-z^p) \in \ZZ_p[pz]$; so Lemma~\ref{L:tilted degree estimate I} shows the first factor has tilted degree $\leq \lfloor \frac jp\rfloor\leq j-1$. The second factor has tilted degree $\leq m-j$ by induction.  By Lemma~\ref{L:tilted degree additive}, the sum has tilted degree $\leq m-1$.  We are done.
\item[(d)]If $(p,n)=(2,4)$, we may proceed as in $(c)$ to show directly that $\binom{2az^4}m$ has tilted degree $\leq m$. The only non-trivial part is the inductive step: we need to show that $\binom{2a(z+1)^4-2az^4}{j}$ has tilted degree $\leq j-1$ for $j\in\mathbb{N}$. To this end, note that 
$2a(z+1)^4-2az^4=a(8z^3+12z^2+8z+2)\in\ZZ_2[2z]$, and then apply Lemma~\ref{L:tilted degree estimate I}.
\end{itemize} 
\end{proof}

\begin{proposition}
\label{P:combined estimate}
\begin{enumerate}
\item
When $ \delta_p = \big(\begin{smallmatrix} a&b\\c&d \end{smallmatrix}\big) \in \big(\begin{smallmatrix} p\ZZ_p&\ZZ_p\\q\ZZ_p&\ZZ_p^\times\end{smallmatrix}\big)$, the coefficient $P_{m,n}(\delta_p)$
belongs to 
\[
\gothm_\Lambda
^{\max\{m-\lfloor n/p\rfloor, 0\}}.
\]
\item
When $ \delta_p = \big(\begin{smallmatrix} a&b\\c&d \end{smallmatrix}\big) \in \bfM_1$, the coefficient
$P_{m,n}(\delta_p)$ belongs to 
\[
 \gothm_\Lambda^{\max\{m-n, 0\}}.
 \]
\end{enumerate}

\end{proposition}
\begin{proof}
The coefficient $P_{m,n}(\delta_p)$ for sure belongs to $\Lambda$ because the operator $||^{[-]}_{\delta_p}$ preserves $\calC(\ZZ_p ; \Lambda)$.
This explains the max on the exponents.

Put $f(z) = (az+b)/(cz+d) \in \ZZ_p\llbracket z\rrbracket$.
Write $d = d_0 \cdot \langle d \rangle$ with $d_0 \in \Delta = (\ZZ/q\ZZ)^\times $ and $\langle d\rangle \in (1+q\ZZ_p)^\times$, so that $[cz+d]$ can be written as
\[
[d_0] \cdot \big[(cz+d)/d_0\big] = [d_0] \cdot (1+T)^{\log((cz+d)/d_0)/q}.
\]
Put $g(z) = \log((cz+d)/d_0)/q$; it is of the form considered in Lemma~\ref{L:tilted degree estimate II}.
Now we have
\[
\tilde \Delta^{(m)} \bigg(
\binom{(az+b)/(cz+d)}{n}\big[cz+d\big]
\bigg)\bigg|_{z = 0}  = [d_0] \cdot \sum_{r \geq 0} T^r \cdot \tilde \Delta^{(m)}\bigg(
\binom{f(z)}{n}\binom{g(z)}r 
\bigg)\bigg|_{z = 0}.
\]
\begin{itemize}
\item
By Lemma~\ref{L:tilted degree estimate II}, $\binom{g(z)}r$ has tilted degree $\leq r$.

\item
In case (1), we have $f(z) \in \ZZ_p\llbracket pz\rrbracket$. So by Lemma~\ref{L:tilted degree estimate I}, 
$\binom {f(z)}n$ has tilted degree $\leq \lfloor \frac np\rfloor$.

\item
In case (2), note that $f(z)$ is of the form considered in Lemma~\ref{L:tilted degree estimate II}. Thus $\binom {f(z)}n$ has tilted degree $\leq n$.
\quash{we can write $f(z)$ as $f_1(z) + f_2(z)$, where $f_1(z) = az/(cz+d)$ is of the form considered in Lemma~\ref{L:tilted degree estimate II} and $f_2(z) = b/(cz+d) \in \ZZ_p\llbracket pz\rrbracket$ is of the form considered in Lemma~\ref{L:tilted degree estimate I}.
We may then use the binomial identity~\eqref{E:binomial identity} to break up $\binom{f_1(z) + f_2(z)}n$ into the sum of the products $\binom{f_1(z)}i \binom{f_2(z)}{n-i}$ for $i=0, \dots, n$.
Then we can combine the bound in Lemma~\ref{L:tilted degree estimate I} (for  $\binom{f_2(z)}{n-i}$) and in Lemma~\ref{L:tilted degree estimate II} (for  $\binom{f_1(z)}i$) using Lemma~\ref{L:tilted degree additive}, to see that $\binom {f(z)}n$ has tilted degree $\leq n$.}
\end{itemize}

Using Lemma~\ref{L:tilted degree additive} again, we see that  the tilted degree of $\binom{f(z)}n \binom{g(z)}r$ is $\leq r + \lfloor \frac np\rfloor$ in case (1), and is $\leq r + n$ in case (2).
So the $T^r$-coefficients of \eqref{E:Pmn} has valuation at least
\[
m - \Big\lfloor \frac np\Big\rfloor - r \textrm{ in case (1)}, \quad \textrm{and}\quad m-n-r \textrm{ in case (2)}.
\]
This is exactly what we need to prove.
\end{proof}

\quash{\begin{remark}
\label{R:Up not compact}
Note that the bound on the coefficients $P_{m,n}(\delta_p) \in \gothm_\Lambda^{\max\{m-\lfloor n/p\rfloor, 0\}}$ for $\delta_p \in \big(\begin{smallmatrix} p\ZZ_p&\ZZ_p\\p\ZZ_p&\ZZ_p^\times\end{smallmatrix}\big)$ does \emph{not} imply that the $U_p$-action on $S^{D}_\inte$ is compact (in the sense of Definition~\ref{D:compact operator}).  For example, modulo $\gothm_\Lambda$, the bound does not show that $P_{m,n}(\delta_p)$ has finite rank.
In fact, it seems that the $U_p$-operator may not be compact (see Example~\ref{Ex:Up not compact}).
We will explain an approach to handle this technical issue in the last section.
\end{remark}}

\begin{lemma}
The rigid space associated to the ring 
\[
\Lambda^{>1/p}: = \Lambda\llbracket pT^{-1}\rrbracket = \ZZ_p\llbracket T, pT^{-1} \rrbracket \otimes_{\ZZ_p} \ZZ_p [ \Delta]
\]
is $\calW^{>1/p}$.
The ideal $\gothm_\Lambda \Lambda^{>1/p}$ is the same as the principal ideal $(T)$.
\end{lemma}
\begin{proof}
This is clear, noting that $p = pT^{-1} \cdot T$.
\end{proof}

In the following, we will work over $\Lambda^{>1/p}$ instead.

\begin{theorem}
\label{T:HP lower bound}
Recall that $t = \# D^\times \backslash (D \otimes \AAA_f)^\times / K^p\Iw_q$.
Let $P = (P_{m,n})_{m,n\in \ZZ_{\geq 0}}$ denote the infinite matrix of the $U_p$-action with respect to the basis
\[
1_0, \dots, 1_{t-1}, z_0, \dots, z_{t-1}, \tbinom {z_0}2, \dots, \tbinom{z_i}{2}, \dots,
\]
as in Proposition~\ref{P:Up operator agrees}. The characteristic power series 
\[
\Char(P): = \lim_{n \to \infty} \det\big(1- X(P_{i,j})_{i,j = 0, \dots, n-1}\big) =  \sum_{n \geq 0}c_n X^n  \in \Lambda\llbracket X \rrbracket\]
 is well defined.
Moreover, we have
\begin{equation}
\label{E:HP bound}
c_n \in T^{\lambda(n)} \cdot \Lambda^{>1/p} \textrm{ for }n \in \ZZ_{\geq 0},\footnote{In the very  recent preprint \cite{johansson-newton} of Johansson and Newton, they gave a more conceptual proof of the estimate in this Theorem.}
\end{equation}
where $\lambda(0) = 0, \lambda(1), \dots$ is a sequence of integers determined by
\[
\lambda(i+1) - \lambda(i) = \Big\lfloor \frac it\Big\rfloor - \Big\lfloor \frac i{pt}\Big\rfloor.
\]
\end{theorem}
\begin{proof}
Combining the estimate in Proposition~\ref{P:combined estimate} and the explicit description of the $U_p$-operator in Proposition~\ref{P:explicit Up}, we see that the $(m,n)$-entry of the infinite matrix $P$ satisfies 
\begin{equation}
\label{E:bound Pmn}
P_{m,n} \in
\gothm_\Lambda^{\max\{\lfloor m/t\rfloor - \lfloor n/pt\rfloor, 0\}}.
\end{equation}
In particular, modulo $\gothm_\Lambda^r$ for each $r\in \mathbb{N}$, the infinite matrix is strict upper triangular except the first $\lfloor ptr/(p-1)\rfloor\times\lfloor ptr/(p-1)\rfloor$-minor. Then modulo $\gothm_\Lambda^r$, the characteristic polynomial of the first $s\times s$-minor is the same as the characteristic polynomial of the first $\lfloor ptr/(p-1)\rfloor\times\lfloor ptr/(p-1)\rfloor$-minor for $s\geq\lfloor ptr/(p-1)\rfloor$. This implies that $\Char(P) \in \Lambda\llbracket X\rrbracket$ is well defined.

To compute $\Char(P)$, we work with a bigger coefficient ring  $\Lambda^{>1/p}$.
So we have
\begin{equation}
\label{E:inverting T}
\gothm_\Lambda^a \cdot T^b \Lambda^{>1/p} = T^{a+b}\Lambda^{>1/p} \textrm{ for any integers }a, b \textrm{ such that } a, a+b \in \ZZ_{\geq 0}.
\end{equation}
We now conjugate the matrix $P$ by the infinite diagonal matrix whose diagonal entries are
\[
\underbrace{1,\dots, 1}_{t}, \underbrace{T, \dots, T}_{t}, \underbrace{T^2, \dots, T^2}_{t}, \dots;
\]
let $P' = (P'_{m,n})_{m,n \in \ZZ_{\geq 0}}$ denote the matrix we get this way.
Then we have
\[
P'_{m,n} \in \gothm_\Lambda^{\max\{\lfloor m/t\rfloor - \lfloor n/pt\rfloor, 0\}} \cdot T^{\lfloor n/t\rfloor - \lfloor m/t\rfloor} \Lambda^{>1/p} \stackrel{\eqref{E:inverting T}} \subseteq T^{\max\{\lfloor n/t\rfloor - \lfloor n/pt\rfloor,\lfloor n/t\rfloor - \lfloor m/t\rfloor \}}\Lambda^{>1/p}.
\]
In particular, the entries of $P'$ in the $n$-th column all lie in $T^{\lfloor n/t\rfloor - \lfloor n/pt\rfloor}\Lambda^{>1/p}$.
So $\Char (P) = \Char(P')$ has the property given in \eqref{E:HP bound}.
\end{proof}

\begin{remark}
Comparing with \cite{wan-xiao-zhang}, the major advantage of our estimate is that the basis we choose in this paper allows us to extend the estimate to the entire annuli, as opposed to just small disks near the boundary of weight space in \cite{wan-xiao-zhang}.
Moreover, we shall see later that the estimate in Theorem~\ref{T:HP lower bound} is already \emph{sharp} for infinitely many $n$.  This magical fact allows us to deduce the strong Theorem~\ref{T:main theorem}.
\end{remark}

By Proposition~\ref{P:Up operator agrees} (where the missing condition is checked by Theorem~\ref{T:HP lower bound}), the characteristic power series of $U_p$ on the space of overconvergent automorphic forms $S^{D, \dagger, m}_{[-]_m}$ (for any $m$) is $\sum_{n \geq 0}c_nX^n \in \Lambda\llbracket X\rrbracket$ with $c_n$ bounded as in \eqref{E:HP bound}.

Now, we fix a character $\omega$ of $\Delta$.
By abuse of notation, we still use $c_n$ to denote its image under the quotient map $\Lambda \to \ZZ_p\llbracket T\rrbracket$ by evaluating $\Delta$ using $\omega$. Write $c_n(T) = \sum_{m \geq 0} b_{n,m} T^m$ for $b_{n,m} \in \ZZ_p$.

\begin{corollary}
\label{C:lower bound}

For $T\in\mathbb{C}_p$ with $0<v(T)<1$,  we have 
$v(c_n(T))\geq\lambda(n)v(T)$ for every $n\geq0$, with equality holding if and only if $b_{n,\lambda(n)}\in\ZZ_p^\times$. Moreover, if $b_{n,\lambda(n)}\notin\ZZ_p^\times$, then 
\[
v(c_n(T))\geq\lambda(n)v(T)+\min\{v(T), 1-v(T)\}.
\] 
As a consequence, the Newton polygon of $\sum_{n\geq0}c_n(T)X^n$ always lies above the polygon with vertices $(n,\lambda(n)v(T))$ for all $n\geq0$.  
\end{corollary}
\begin{proof}

First note that if $\sum_{m\in\ZZ}d_m T^m\in\Lambda^{>1/p}$, then $v(d_m)\geq\max\{0,-m\}$.
Combining this fact with (\ref{E:HP bound}), we get
\[
v(b_{n,m}) \geq \max\{ \lambda(n)-m, 0\}.
\]
For $T \in \mathbb{C}_p$ with $0<v(T) <1$, we deduce 
\begin{equation}
\label{E:lower bound}
v(b_{n,m} T^m) \geq \max\{ \lambda(n)-m, 0\} +m v(T)\geq \lambda(n)v(T),
\end{equation}
with the second equality holding if and only if $m=\lambda(n)$. It follows that we always have 
$v(b_{n,m} T^m) \geq \lambda(n)v(T)$, 
with equality holding if and only if $m=\lambda(n)$ and  $b_{n, \lambda(n)}$ is a $p$-adic unit in $\ZZ_p$. The rest of the corollary is clear. 
\end{proof}

We call the polygon given in Corollary \ref{C:lower bound} the \emph{lower bound polygon} of $\sum_{n\geq0}c_n(T)X^n$. 

In the next we consider the ordinary locus of the (entire) spectral curve. In fact, it is very well known among the experts that the ordinary locus of the eigencurve is nothing but the Hida family of ordinary modular forms.  Consequently, it should be finite and flat over weight space. Because of the lack of proper reference, we give a proof about this fact in our case as follows.
\begin{theorem}
\label{T:hida}

Let $X^\ord$ denote the locus of the spectral curve $\Spc_D$ where $a_p(x)$ has valuation zero.
Then the following holds:
\begin{enumerate}
\item 
$X^\ord$ is finite and flat over $\calW$;
\item
$X^\ord$ is a union of connected components of $\Spc_D$.
\end{enumerate}
\end{theorem}
\begin{proof}
Applying Remark \ref{R:non-neat}, it reduces to the case that tame level is neat. Then it suffices to prove this over each connected component of $\calW$; so we fix a character $\omega$ of $\Delta$ throughout and denote by $[-]_\omega$ the universal character.
Let $X^\ord_{\omega}$ denote the corresponding components.
Recall that the characteristic power series of $U_p$ on $S^{D, \dagger}_{[-]_\omega}$ is equal to $\sum_{n=0}^\infty c_n(T)X^n$.
Let $d$ be the maximal index such that $c_d(T)$ is a unit in $\ZZ_p\llbracket T\rrbracket$, or equivalently, the constant term of $c_d(T)$ is a $p$-adic unit in $\ZZ_p$; such a $d$ must exist by Corollary \ref{C:lower bound}.
We claim that $X^\ord_{\omega}$ is finite and flat of degree $d$ over $\calW_{\omega}$.
Indeed, for each weight character $\chi: \ZZ_p^\times \to \CC_p^\times$ whose restriction to $\Delta$ is $\omega$, the Newton polygon of $\sum_{n=0}^\infty c_n(T_\chi)X^n$ has slope zero in the first $d$ segment and the vertices with $x$-coordinate strictly bigger than $d$ must have $y$-coordinate at least $\min\{1, v(T_\chi)\}$, and the corresponding slopes eventually tend to infinity.  This implies that  $X^\ord_{\omega}$ is finite and flat over $\calW_{\omega}$ of degree $d$, and is an affinoid subdomain when restricted to the fiber of each affinoid subdomain of $\calW_{\omega}$. Moreover, the first non-zero slope of the Newton polygon at $\chi$ can be bounded away from $0$ uniformly over any affinoid subdomain; so $X^\ord$ is disconnected from its complement.
\end{proof}

\begin{remark}
For a classical weight $(k,\psi)$, its fiber in $X^\ord$  exactly corresponds to the ordinary part of the space of automorphic forms $S^{D,\dagger}_{(k,\psi)}$.  By the Zariski density of classical weights in  weight space, one can show that $X^\ord$ is the same as the spectral Hida family (e.g. by the standard control theorem for Hida families (\cite[Theorem 7.1(5)]{hida})).
 \end{remark}

Recall that $r_\ord(\omega)$ denotes the dimension of the ordinary subspace of automorphic forms of weight $2$ and character $\omega$.

\begin{corollary}[Hida]
\label{C:hida}
The degree of $X^\ord_{\omega}$ over $\calW_{\omega}$ is $r_{\ord}(\omega)$. As a consequence, for any integer $k\in \ZZ$ and  any finite character $\psi$ of conductor $p^m$, the slope zero subspace of $S^{D,\dagger}_{(k,\psi)}$ has dimension
$r_{\ord}(\psi|_\Delta \cdot \omega_0^{k}).$
\end{corollary}
\begin{proof}  
In the course of the proof of Theorem \ref{T:hida}, we already saw that the degree of $X^\ord_{\omega}$ over $\calW_{\omega}$ is equal to the dimension of the ordinary subspace of $S^{D,\dagger}_{\chi}(K^p)$ for every weight character $\chi\in\calW_{\omega}$. In particular, we can choose $\chi=\omega$.  The rest part follows from the fact that $(k,\psi)$ belongs to the weight disk corresponding to $\psi|_\Delta \cdot \omega_0^{k}$.
\end{proof}

Let $m\in\NN_{\geq2}$, and let $\psi$ be a finite character of conductor $p^m$. For $k \in \ZZ_{\geq 0}$, using an isomorphism analogous to \eqref{E:explicit integral model}, we see that $S^D_{k+2}(K^p\Iw_{p^m};\psi)$ is isomorphic to 
the direct sum of $t$ copies of $\LP^{m-v(q), \deg \leq k}(\ZZ_p; E)$.  So in total, 
\begin{equation}\label{E:atkin-lehner}
\dim S^D_{k+2}(K^p\Iw_{p^m}; \psi) = (k+1)q^{-1}p^mt.
\end{equation}
To prove Theorem \ref{T:main theorem}, we also need the following result on Atkin--Lehner theory. 

\begin{proposition}[Atkin--Lehner]
\label{P:atkin-lehner}
We use $\alpha_0(\psi), \dots, \alpha_{(k+1)q^{-1}p^mt-1}(\psi)$ to denote the slopes of $U_p$ acting on $S^D_{k+2}(K^p\Iw_{p^m}, \psi)$ in non-decreasing order.
Then we have
\[
\alpha_i(\psi) = k+1-\alpha_{(k+1)q^{-1}p^mt-1-i}(\psi^{-1}) \quad \textrm{for }i =0, \dots, (k+1)q^{-1}p^mt-1.
\]
In particular,
the total sum of the $U_p$-slopes of $S^D_{k+2}(K^p\Iw_{p^m}; \psi) \oplus S^D_{k+2}(K^p\Iw_{p^m}; \psi^{-1})$ is
 $(k+1)^2q^{-1}p^mt$.
\end{proposition}
\begin{proof}
 This is well known to the experts.
We here give a proof as follows. Firstly note that the base change of $S^D_{k+2}(K^p\Iw_{p^m};\psi)$ to $\CC$ is isomorphic to the corresponding classical space of automorphic forms for $D$. Since $\psi$ has conductor $p^m$ while the level structure at $p$ is $\Iw_{p^m}$, by applying Jacquet--Langlands and \cite[Proposition 2.8]{loeffler-weinstein}, we see that for every automorphic representation $\pi$ appearing in $S^D_{k+2}(K^p\Iw_{p^m};\psi)$, its $p$-component $\pi_p$ is a principal series of $\GL_2(\QQ_p)$ whose corresponding two characters of $\QQ_p^\times$ are $\unr(\alpha)$ and $\unr(\alpha^{-1}) \otimes \omega_p$, where $\unr(?)$ is an unramified character of $\QQ_p^\times$ sending $p$ to $?$ and $\omega_p$ is the  $p$-component of the Hecke character associated to  $\psi$. 
Moreover, the $U_p$-eigenvalue on the $\Iw_{p^m}$ fixed vector of $\pi_p$ is $\alpha p^{(k+1)/2}$.
But we can twist the representation $\pi$ by a central Hecke character associated to $\psi^{-1}$; then the resulting automorphic representation would appear in $S^D_{k+2}(K^p\Iw_{p^m};\psi^{-1})$ (or rather its base change to $\CC$) instead.
Moreover, the $U_p$-eigenvalue on the $\Iw_{p^m}$ fixed vector of the $p$-component of this automorphic representation is $\alpha^{-1}p^{(k+1)/2}$.
In conclusion, one can pair the $U_p$-eigenvalues of $S^D_{k+2}(K^p\Iw_{p^m}; \psi)$ and the $U_p$-eigenvalues of $ S^D_{k+2}(K^p\Iw_{p^m}; \psi^{-1})$ so that they multiply to $p^{k+1}$.
Our assertion on slopes follows from this.
\end{proof}

\subsection{Proof of Theorem~\ref{T:main theorem}} 
\label{SS:Buzzard-Kilford D}
By virtue of Remark \ref{R:non-neat}, we only need to treat the case that the tame level is neat. 
To this end, it suffices to get the desired decomposition for each $\Spc^{>1/p}_{D,\omega}$. We consider the classical weights $\chi_k = (k, \psi)$ of conductor 
$q^2$ with $k \in \ZZ_{\geq 0}$, such that $\chi_k|_{\Delta} = \omega$. The corresponding $T$-coordinates $T_{\chi_k}$ have valuation $\frac{q}{\varphi(q^2)}=\frac{p}{q(p-1)}<1$. The space of overconvergent automorphic forms $S^{D,\dagger}_{\chi_k}$ contains the subspace $S^D_{k+2}(K^p\Iw_{q^2};\psi)$ of classical automorphic forms. Put $n_k=kqt$ for $k\in\ZZ_{\geq0}$. By (\ref{E:atkin-lehner}), we have
\[
\dim S^D_{k+2}(K^p\Iw_{q^2}; \psi) = (k+1)qt=n_{k+1}.
\]

{\bf Step I:} The first important observation is that the Newton polygon of $\sum_{n\geq0}c_n(T_{\chi_k})X^n$ touches the lower bound polygon at the points 
\[
\calP_k:=\big(n_{k+1}, \lambda(n_{k+1})v(T_{\chi_k})\big).
\] 
On the one hand, Proposition~\ref{P:atkin-lehner} says that the sum of  all $U_p$-slopes on $S^D_{k+2}(K^p\Iw_{q^2}; \psi) \oplus S^D_{k+2}(K^p\Iw_{q^2}; \psi^{-1})$ is $(k+1)^2qt$. 
By Proposition~\ref{P:classicality}, one deduces that the set of all $U_p$-slopes on $S^D_{k+2}(K^p\Iw_{q^2}; \psi) \oplus S^D_{k+2}(K^p\Iw_{q^2}; \psi^{-1})$ is  exactly the set of the first $n_{k+1}$ $U_p$-slopes in each of $S^{D, \dagger}_{(k, \psi)}$ and $S^{D, \dagger}_{(k, \psi^{-1})}$. It follows that the sum of the first $n_{k+1}$ $U_p$-slopes in each of $S^{D, \dagger}_{(k, \psi)}$ and $S^{D, \dagger}_{(k, \psi^{-1})}$ is also $(k+1)^2qt$.
  
On the other hand, for each of  $T_{(k,\chi)}$ and $T_{(k,\chi^{-1})}$, when the $x$-coordinate is $(k+1)qt$, the $y$-coordinate of the lower bound polygon is
\begin{equation}
\label{E:lambda(n_k)}
\begin{split}
\frac p{q(p-1)}\lambda\big((k+1)qt\big) = &\frac p{q(p-1)} \sum_{n=0}^{(k+1)qt-1} \Big(\Big\lfloor \frac nt \Big\rfloor - \Big\lfloor \frac n{pt} \Big\rfloor \Big) 
\\=\ &
\frac p{q(p-1)} \bigg( t\sum_{n=0}^{(k+1)q-1} n - pt \sum_{n=0}^{\frac{(k+1)q}{p}-1} n  \bigg) \\=\ & \frac{(k+1)^2qt}2.
\end{split}
\end{equation}
This exactly agrees (!) with half of the sum of the first $n_{k+1}$ $U_p$-slopes on each of 
$S^{D, \dagger}_{(k, \psi)}$ and $S^{D, \dagger}_{(k, \psi^{-1})}$.
In particular, we see that the sum of first $n_{k+1}$ $U_p$-slopes on $S^{D, \dagger}_{(k, \psi)}$ (resp. $S^{D, \dagger}_{(k, \psi^{-1})}$) is $(k+1)^2qt/2$. That is, the Newton polygon of $\sum_{n\geq0}c_n(T_{\chi_k})X^n$ passes through the point 
\[
\big((k+1)qt, (k+1)^2qt/2\big)=\big(n_{k+1}, \lambda(n_{k+1})v(T_{\chi_k})\big).
\]
We admit that the matching of two sums of $U_p$-slopes is quite lucky in our case, which we do not know how to reproduce in too much more generality.
From now on, we may work only with $S^{D, \dagger}_{(k, \psi)}$. (For each $k$, we fix one $\psi$ of conductor $q^2$ that satisfy $\omega_0^k\psi|_\Delta = \omega$.)

{\bf Step II:} We deduce the decomposition of $\Spc_D^{>1/p}$ from the touching of polygons.

To proceed, first note that $\lfloor \frac{n_k}t\rfloor - \lfloor \frac{n_k}{pt}\rfloor = kq-kqp^{-1} = k\varphi(q)$. So if $i\in\ZZ$ and $n_{k+1}-i\geq0$, then 
\[
\lambda(n_{k+1}-i)\geq\lambda(n_{k+1})-(k+1)\varphi(q)i
\]
with equality if and only if $i\in[-t,t]$. 
Since $|T_{\chi_k}| = p^{-p/q(p-1)} \in (1/p, 1)$, by Corollary \ref{C:lower bound}, we have 
\begin{equation}
\label{E:must be unit}
v(c_n(T_{\chi_k})) = \lambda (n)v(T_{\chi_k}) \textrm{  if and only if } b_{n, \lambda(n)} \textrm{ is a $p$-adic unit in }\ZZ_p.
\end{equation}
We have previously shown that the Newton polygon of $\sum_{n \geq 0} c_n(T_{\chi_k})X^n$ exactly goes through the point 
$\calP_k=\big(n_{k+1}, \lambda(n_{k+1})v(T_{\chi_k})\big)$.
Note that this does not force $b_{n_{k+1}, \lambda(n_{k+1})}$ to be a $p$-adic unit, as the point $\calP_k$ may not be a vertex for the Newton polygon.
We thus suppose that the line segment of the Newton polygon of $\sum_{n \geq 0} c_n(T_{\chi_k}) X^n$ passing through $\calP_k$ lies over $[n_{k+1}^-, n_{k+1}^+]$ in its $x$-coordinate. It is clear that if $n_{k+1}^-\neq n_{k+1}^+$, then this line segment has slope $(k+1)\varphi(q)$. 
It follows that 
\[
n_{k+1}^- \in [n_{k+1}-t, n_{k+1}] \quad \textrm{and} \quad n_{k+1}^+ \in [n_{k+1}, n_{k+1}+t].
\]
Moreover, the equivalence \eqref{E:must be unit} implies that  $n_{k+1}^-$ (resp. $n_{k+1}^+$) is the minimal index in $[n_{k+1}-t, n_{k+1}] $ (resp. maximal index  in $[n_{k+1}, n_{k+1} + t]$) such that
\[
b_{n_{k+1}^-, \lambda(n_{k+1}^-)} \textrm{ (resp. }b_{n_{k+1}^+, \lambda(n_{k+1}^+)} \textrm{) is a $p$-adic unit in }\ZZ_p.
\]
For a uniform treatment later, we set $n_0^- =0$ and $n_0^+$ the maximal index in $[0, t]$ such that $b_{n_0^+, 0}$ is a $p$-adic unit.

Now, if we specialize to any point $T\in \calW_\omega^{>1/p}$, we must have (for all $i\in \ZZ_{\geq0}$)
$$v(c_{n_k-i}(T)) \geq v(T) \lambda(n_k-i) \geq  v(T) \cdot \big( \lambda(n_k) - k\varphi(q)i \big),$$
where the first inequality is a strict inequality if $n_k-t\leq n_k - i < n_k^- $ (by the minimality of $n_k^-$) and the second inequality is a strict inequality if $n_k-i<n_k-t$. 
In summary,  for $i\in \ZZ_{\geq0}$, we have the inequality 
$$v(c_{n_k-i}(T)) \geq v(T) \cdot \big( \lambda(n_k) - k\varphi(q)i \big),$$
which becomes a strict inequality if $n_k-i< n_k^{-}$ and becomes an equality if $n_k-i=n_k^-$. Similarly, we have the inequality 
\begin{align*}
v(c_{n_k+i}(T))\geq v(T) \cdot \big( \lambda(n_k) + k\varphi(q)i \big), 
\end{align*}
which becomes a strict inequality if $n_k+i>n_k^{+}$ and becomes an equality if $n_k+i=n_k^+$. 
Moreover, by Corollary \ref{C:lower bound}, we see that the differences in all strict inequalities are at least $\min \{v(T), 1-v(T)\}$.

In summary, we conclude that for every $T\in\mathbb{C}_p$ with $0<v(T)<1$, if $n_k^-\neq n_k^+$, then the points  
\[
\big(n_{k}^-,\lambda(n_{k}^-)v(T)\big)\text{ and }\big(n_{k}^+,\lambda(n_{k}^+)v(T)\big)
\]
are two consecutive vertices of the Newton polygon of $\sum_{n\geq0}c_n(T)X^n$. Furthermore, the line segment connecting these two vertices has slope $k\varphi(q)v(T)$, and passes through the point $\big(n_{k},\lambda(n_{k})v(T)\big)$. Otherwise, $n_k^-=n_k=n_k^+$ is a vertex of the Newton polygon of $\sum_{n\geq0}c_n(T)X^n$.

The decomposition of the spectral curve follows from this.
More precisely, for $I = k = [k,k]$ or $(k, k+1)$ with $k \in \ZZ_{\geq 0}$, we define $X_{I,\omega}$ to be the open subspace of $\Spc_{D,\omega}^{>1/p}$
such that for each point $z \in X_{I,\omega}$, we have
\[
v(a_p(z)) \in \varphi(q) v(T_{\wt(z)}) \cdot I.
\]
By our previous estimates and applying \cite[Corollary 4.3]{buzzard}, these are in fact finite flat over $\calW^{>1/p}_{\omega}$, and are affinoid subdomains when restricted to fibers of any affinoid subdomain of $\calW^{>1/p}_{\omega}$.
It follows that these $X_{I,\omega}$'s are unions of connected components of $\Spc_{D}^{>1/p}$.
Regarding the degrees, we must have
\begin{equation}\label{E:inequality1}
\sum_{j=0}^{k-1}\big(\deg X_{j,\omega} + \deg X_{(j, j+1),\omega} \big) = n_k^- \in [n_k -t, n_k] \textrm{ and}
\end{equation}
\begin{equation}\label{E:inequality2}
\sum_{j=0}^{k-1}\big(\deg X_{j,\omega} + \deg X_{(j, j+1),\omega} \big) + \deg X_{k,\omega}
= n_k^+ \in [n_k, n_k+t].
\end{equation}

{\bf Step III:} It remains to compute the degrees of $X_{I,\omega}$'s. It is clear that $X_{0, \omega}$ coincides with the restriction of $X^\ord_{\omega}$, which is introduced in the proof of Theorem \ref{T:hida},  to $\calW_\omega^{>1/p}$. Then by Corollary~\ref{C:hida}, $\deg X_{0,\omega}$ is equal to the dimension of slope zero subspace in $S^{D, \dagger}_{\omega}$. That is,
\[
\deg X_{0,\omega}=n_0^+=r_{\ord}(\omega).
\]
(One subtlety of the argument here is that: we can not directly apply the previous part of Theorem \ref{T:main theorem} because the weight $\omega\notin \calW^{>1/p}$. We have to employ Corollary~\ref{C:hida} instead.)
For $k\geq0$, first note that $n_{k+1}- n_{k+1}^-$ is equal to the dimension of slope $k+1$ subspace in $S^{D}_{k+2}(K^p\Iw_{q^2}; \psi)$.
By Atkin--Lehner theory (Proposition \ref{P:atkin-lehner}) and Proposition \ref{P:classicality}, the multiplicity is the same as the dimension of slope zero subspace in $S^{D, \dagger}_{(k,\psi^{-1})}$. Using Corollary~\ref{C:hida} again, we deduce that 
\[
n_{k+1}-n_{k+1}^-=r_{\ord}\big(\psi^{-1}|_{\Delta}\cdot\omega_0^k\big)=r_{\ord}\big(\omega^{-1}\omega_0^{2k}\big)
\] 
because $\psi^{-1}|_\Delta \cdot \omega_0^k=\chi_k^{-1}|_{\Delta}\cdot\omega_0^{2k}=\omega^{-1}\omega^{2k}_0$. 

To compute $n_{k+1}^+ - n_{k+1}$, we recall the following exact sequence (cf. \cite{jone})
\[
0 \to S^D_{k+2}(K^p\Iw_{q^2}, \psi) \to S^{D, \dagger}_{(k,\psi)} \xrightarrow{\left(\frac{d}{dz}\right)^{k+1}} S^{D, \dagger}_{(-k-2, \psi)} \to 0.
\]
This exact sequence is equivariant for the $U_p$-action on the first two spaces, and the $p^{k+1}U_p$-action on the third space.
It is clear that $n_{k+1}^+-n_{k+1}$ is equal to the codimension of $S^D_{k+2}(K^p\Iw_{q^2}, \psi)$ in the slope $\leq k+1$ subspace in $S^{D, \dagger}_{(k,\psi)}$. The latter in turn is equal to the dimension of slope zero subspace of $S^{D, \dagger}_{(-k-2, \psi)}$ by the exact sequence.
Using Corollary~\ref{C:hida}, we thus obtain
\[
n_{k+1}^+ - n_{k+1} = r_\ord\big(\psi|_\Delta \cdot \omega_0^{-k-2}\big)=r_{\ord}\big(\omega\omega_0^{-2k-2}\big).
\]

The final degree is computed by
\[
\deg X_{k,\omega} = n_k^+ - n_k^-=(n_k^+-n_k)+(n_k-n_k^{-})=\left\{
\begin{array}{l}
r_{\ord}(\omega), \hspace{39.5mm}\textrm{if}\hspace{2mm}k=0,\\
r_{\ord}(\omega^{-1}\omega_0^{2k-2})+r_{\ord}(\omega\omega_0^{-2k}), \textrm{if}\hspace{2mm}k\geq1,\\
\end{array}
\right.
\]
and
\begin{align*}
\deg X_{(k,k+1),\omega} = n_{k+1}^- - n_k^+&=n_{k+1}-n_k-(n_{k+1}-n_{k+1}^{-})-(n_k^+-n_k)\\
&=qt-r_{\ord}(\omega^{-1}\omega_0^{2k})-r_{\ord}(\omega\omega_0^{-2k}).
\end{align*}
This concludes the proof of Theorem~\ref{T:main theorem}.
\hfill $\Box$

The following interesting consequence of Theorem \ref{T:main theorem} is pointed out to us by Chenevier. We are grateful to him for allowing us to include it in this paper. 

 \begin{proposition}
\label{P:chenevier} 
Let C be an irreducible component of $\Spc_D$.
If C is finite over weight space, then C is inside the ordinary locus.
\end{proposition}
\begin{proof}
Under the assumption, the weight map $C\rightarrow\calW$ is finite and flat (the flatness is ensured by \cite[Theorem C]{coleman-mazur}).
Thus the analytic function $a_p$ on $C$, which is nowhere vanishing and bounded by 1, has a norm $g$ down to some weight disk $\calW_\omega$.
It is clear that the analytic function $g$ is also a nowhere vanishing and bounded by 1, so it has the form $p^n h$ where $h$ is a unit in $\ZZ_p\llbracket T\rrbracket$.
In particular, this shows that for all $w\in\calW_\omega$,  
\[
\sum_{x \in C, \wt(x)=w} v(a_p(x)) = n. 
\]
But Theorem \ref{T:main theorem} says that $v(a_p(x))$ goes to $0$ above the boundary of $\calW_\omega$, so $n=0$. Thus $v(a_p(x))=0$ for all $x\in C$, concluding the proposition.
\end{proof}

\begin{remark}
\label{R:integral model of spectral curve}
We note that the existence of $n_k^\pm$ in the proof of Theorem~\ref{T:main theorem} in fact implies that, for $n = n_k^\pm$, 
$c_n(T)$ is equal to $T^{\lambda_n}$ times a unit in $\Lambda^{>1/p}$.
Then, a standard factorization argument shows (see e.g. \cite[Proposition~3.2.2]{kedlaya} for the argument) that we can factor $\Char(P)$ into the following product
\[
P_0(X) \cdot P_{(0,1)}(X)\cdot P_1 (X)\cdot P_{(1,2)}(X) \cdots
\]
such that each $P_I(X) \in \Lambda^{>1/p}\llbracket X\rrbracket$ is the characteristic polynomial corresponding to the component $X_I$.
This gives an integral model $\gothX_I$ of each $X_I$ (in the case that the tame level is neat).

An intriguing and pressing future question is: what is the arithmetic property at the ``special fibers" of $\gothX_n$ and $\gothX_{(n,n+1)}$?
In particular, the pseudo-representation of $\Gal(\overline \QQ/\QQ)$ can be extended to the integral model; what can we say about the representation over the special fiber of these formal schemes?  We hope to come back to this question in a future work.  At the same time, we encourage the readers to explore more applications of Coleman's idea on studying the integral model of the eigencurve.

An alternative way to understand the integral model of the eigencurve is to ``compactify" the weight space in the category of adic spaces, by viewing the boundary part $\calW^{>1/p}$ in the $T$-adic world and adding a point at the ``boundary" on each disk (whose residue field is $\FF_p((T))$). Then one can extend the spectral curve over these ``boundary points."
We refer to \cite{AIP} and the recent preprint \cite{johansson-newton} for more discussion on this viewpoint.
\end{remark}

\begin{remark}
\label{R:generalization}
We discuss some potential generalization of our main theorem.
\begin{enumerate}
\quash{\item
We excluded the case of $p=2$ merely to simplify the presentation. Our argument should work for $p=2$ with little modification.  Another reason for excluding this case is because we hold faith that our argument should hold for quite general algebraic groups $G$, in which situation, we would be able to treat all prime numbers uniformly.}

\item
As pointed out in Remark~\ref{R:relation to others work}(1), our result cannot access the eigencurve with trivial tame level structure.
There are two possible strategies to remedy this.
One is to work with modular symbols; it might to be possible to replicate the argument in this paper under that setup.  We encourage interested readers to explore this possibility.
Another approach is to base change (in a $p$-adic family) to a real quadratic field $F$ in which $p$ splits.
Then the eigensurface for the unramified definite quaternion algebra $D$ over $F$ should be the same as the eigensurface for overconvergent Hilbert modular forms over $F$.
Then the analogous results for the unramified eigencurve should follow from that of the eigensurface for $D$ (if the latter case may be proved).

\item
Now, take a general algebraic group $G$ over $\QQ$ which is quasi-split at $p$ and whose $\RR$-points are compact modulo center.
Then one can construct the associated eigenvariety $\calC_G$ (as carried out in \cite{loeffler}).
Our ultimate optimistic expectation is that an analogue of Proposition~\ref{P:combined estimate} still holds true.\footnote{This was recently verified by Johansson and Newton \cite{johansson-newton}.} In particular, we can still see the characteristic power series of the $U_p$-operators on the space of integral automorphic forms (with respect to an appropriate basis).
Nonetheless, the classicality argument and the touching of Newton polygon are no longer available, at least not in a  naive way.
We strongly encourage interested readers to investigate in this issue.
\end{enumerate}
\end{remark}

\section{Distribution of $U_p$-slopes}
\label{Sec:distribution of Up}
This section is devoted to proving Theorem \ref{T:main theorem 2}. 
Our argument for the first half of the theorem is modeled on the proof of \cite[Theorem 3.8]{davis-wan-xiao}, whose upshot is to give a suitable upper bound for the Newton polygon. The rest of the theorem, i.e. the arithmetic progression statement about the ratios, then follows easily from Atkin--Lehner theory.  
 
To start with, we first point out an upper bound of the Newton polygon of $\sum_{n=0}^{\infty}c_n(T)X^n$ when $0<v(T)<1$. Indeed, in the course of the proof of Theorem \ref{T:main theorem}, we already show that the Newton polygon of $\sum_{n\geq0}c_n(T)X^n$ passes through the points
$(n_k, \lambda(n_k)v(T))$ for all $k\geq0$. Therefore, we deduce that the Newton polygon of $\sum_{n\geq0}c_n(T)X^n$ always lies below the polygon with vertices 
$(n_k, \lambda(n_k)v(T))$ for all $k\geq 0$. We call this polygon the \emph{upper bound polygon} of $\sum_{n=0}^\infty c_n(T)X^n$. 

\begin{lemma}
\label{L:vertical difference}
The maximal vertical difference between the lower bound polygon and the upper bound polygon is $(p^2-1)tv(T)/8$ for $p>2$, and $tv(T)$ for $p=2$.
\end{lemma}
\begin{proof}
We only treat the case $p>2$, the case $p=2$ being similar. Note that the lower bound polygon and the upper bound polygon touch at the vertices $(n_k,\lambda(n_k)v(T))$ for $k\geq0$. It is sufficient to bound their vertical difference over $x\in[n_k, n_{k+1}]$. By (\ref{E:lambda(n_k)}), we first get 
$\lambda(n_{k+1})=(k+1)^2p(p-1)t/2$. 
A short computation then shows that the restriction of the upper bound polygon on $[n_k, n_{k+1}]$ is a linear function with slope $(k+\frac 1{2})(p-1)v(T)$. On the other hand, for every integer  $a\in[0,p-1]$,  by Theorem \ref{T:HP lower bound}, we know that  the restriction of the lower bound polygon on
$[n_k+at, n_k+(a+1)t]$ is a linear function with slope $(k(p-1)+a)v(T)$. We therefore deduce that the maximal vertical difference over $[n_k, n_{k+1}]$ is achieved when $a=\frac{p-1}{2}$. 
 
 In that case, put $n=n_k+(p-1)t/2$. It is straightforward to see the vertical difference at $x=n$ is, by looking at the incremental differences of slopes built from the vertex $(n_k, \lambda(n_k) v(T))$, 
\begin{align*}
&\sum_{i = n_k}^{n-1}  \bigg( \Big(k+\frac 12\Big)(p-1) - \Big( \Big\lfloor\frac i{t}\Big\rfloor-\Big\lfloor\frac i{pt}\Big\rfloor \Big) \bigg) v(T)
\\
= \ & t v(T) \sum_{j = kp}^{kp + \frac{p-1}{2}-1}  \bigg( \Big(k+\frac 12\Big)(p-1) - \Big( j -\Big\lfloor\frac jp\Big\rfloor \Big) \bigg)
\\
=\ &tv(T)  \sum_{j =0}^{ \frac{p-1}{2}-1}  \Big( \frac {p-1}{2} - j\Big) = \frac 18 (p^2-1) tv(T).\qedhere
\end{align*}
\end{proof}

\subsection{Proof of Theorem~\ref{T:main theorem 2}}
\label{SS:proof of main theorem 2} 
We first show the existence of $\lambda$, the sequence $\alpha_0(\omega),\alpha_1(\omega),\dots$ and the desired decomposition for $\Spc_{D,\omega}^{>\lambda}$. For this purpose, by virtue of Remark \ref{R:non-neat}, it is sufficient to treat the case that the tame level is neat. \quash{Then it suffices to prove the theorem for $\Spc_{D,\omega}^{>\lambda}$ with each $\omega$.} Also, we assume $p>2$, the case $p=2$ being similar. We will proceed as in the proof of Theorem \ref{T:main theorem}. That is, it suffices to show that for $T\in\mathbb{C}_p$ with $0<v(T)<\frac{8}{(p^2-1)t+8}$, the ratios to $v(T)$ of the slopes (counted with multiplicity) of the Newton polygon of $\sum_{n\geq0}c_n(T)X^n$ are independent of the choice of $T$. 

Recall that the Newton polygon of $\sum_{n\geq0}c_n(T)X^n$ is the convex hull of points
 $(n, v(c_n(T)))$ for all $n\geq0$. We consider those points which lie below the upper bound polygon.

{\bf Claim:} If $(l, v(c_l(T_0))$ lies strictly below the upper bound polygon for some $l\in\mathbb{N}$ and $T_0\in\mathbb{C}_p$ with $0<v(T_0)<\frac{8}{(p^2-1)t+8}$, then there exists a \emph{unique} integer $m(l)\geq\lambda(l)$ such that for every $T\in\mathbb{C}_p$ with $0<v(T)<\frac{8}{(p^2-1)t+8}$, $(l, v(c_l(T)))$ lies strictly below the upper bound polygon and $v(c_l(T))=m(l)v(T)$. 
 
Granting the claim, we conclude that there exists a (finite or infinite) set of positive integers $\{l_i\}_{i\in I}$ such that if  $0<v(T)<\frac{8}{(p^2-1)t+8}$, then the Newton polygon of  $\sum_{n\geq0}c_n(T)X^n$ is the convex hull of points 
\[
\big\{(n_k,\lambda(n_k)v(T))\big\}_{k\geq 0}\coprod \big\{(l_i, m(l_i)v(T)\big\}_{i\in I}.
\] 
It is then clear that the ratios to $v(T)$ of the slopes of this polygon are independent of $T$. This yields the existence of the sequence $\alpha_0(\omega),\alpha_1(\omega),\dots$ and the desired decomposition for $\Spc_{D, \omega}^{>\lambda}$. 

We now proceed to show the claim. First note that if $0<v(T)<\frac{8}{(p^2-1)t+8}$ and $m<\lambda(l)$, then by (\ref{E:lower bound}), we get
\begin{equation}
\label{E:m less than lambda(l)}
v(b_{l,m}T^m)\geq \lambda(l)-m+mv(T)\geq \lambda(l)v(T)-v(T)+1>\lambda(l)v(T)+\frac{(p^2-1)tv(T)}{8}.
\end{equation}
On the other hand, since $(l, v(c_l(T_0)))$ lies strictly below the upper bound polygon, by Lemma \ref{L:vertical difference}, we get 
\[
v(c_l(T_0))-\lambda(l)v(T_0)<\frac{(p^2-1)tv(T_0)}{8}.
\] 
Hence for $m<\lambda(l)$, we obtain
\begin{equation}
\label{E:m<lambda(l)}
v(c_l(T_0))<\lambda(l)v(T_0)+\frac{(p^2-1)tv(T_0)}{8}<v(b_{l,m}T_0^m)
\end{equation}
by specializing (\ref{E:m less than lambda(l)}) to $T=T_0$.
Therefore, there must be some $m\geq\lambda(l)$ such that 
$v(b_{l,m}T_0^m)\leq v(c_l(T_0))$. Let $m(l)$ be the minimal one satisfying this property. 
It follows that
\[
 v(b_{l,m(l)})\leq v(b_{l,m(l)}T_0^{m(l)})-\lambda(l)v(T_0)\leq v(c_l(T_0))-\lambda(l)v(T_0)<\frac{(p^2-1)tv(T_0)}{8}<1, 
\]
yielding $b_{l,m(l)}\in\mathbb{Z}_p^\times$. Thus for $m>m(l)$, we get
\begin{equation}
\label{E:m bigger than lambda(l)}
v(b_{l,m}T^m)>v(b_{l,m})+m(l)v(T)\geq m(l)v(T)=v(b_{l,m(l)}T^{m(l)}).
\end{equation}
Moreover, by the minimality of $m(l)$, for $m\in [\lambda(l), m(l)-1]$, we have 
\begin{equation}
\label{E:lambda(l)<m<m(l)}
v(b_{l,m}T_0^{m})>v(c_l(T_0))\geq v(b_{l,m(l)}T_0^{m(l)}),
\end{equation}
yielding $v(b_{l,m})>v(b_{l,m(l)})$. Hence $b_{l,m}\in p\mathbb{Z}_p$ for those $m$. Finally, putting (\ref{E:m<lambda(l)}), (\ref{E:m bigger than lambda(l)}), and (\ref{E:lambda(l)<m<m(l)}) together, we conclude    $v(c_l(T_0))=v(b_{l,m(l)}T_0^{m(l)})=m(l)v(T_0)$.

Now let $0<v(T)<\frac{8}{(p^2-1)t+8}$. Since the point $(l, m(l)v(T_0))$ lies strictly below the upper bound polygon for $T_0$, by similarity, the point $(l, m(l)v(T))$ lies strictly below the upper bound polygon for $T$ as well. Note that \eqref{E:m less than lambda(l)} together with Lemma \ref{L:vertical difference} imply that for $m<\lambda(l)$, the point $(l, v(b_{l,m}T^m))$ lies above the upper bound polygon. Hence $v(b_{l,m}T^m)>m(l)v(T)$ for $m<\lambda(l)$. For $m\in [\lambda(l), m(l)-1]$, since $b_{l,m}\in p\ZZ_p$, it follows that 
\[
v(b_{l,m}T^{m})\geq 1+\lambda(l)v(T)>\frac{(p^2-1)tv(T)}{8}+\lambda(l)v(T).
\] 
Hence $(l, v(b_{l,m}T^{m}))$ lies above the upper bound polygon by Lemma \ref{L:vertical difference}, yielding that $v(b_{l,m}T^m)>m(l)v(T)$. For $m>m(l)$, we have $v(b_{l,m}T^m)>m(l)v(T)$ by (\ref{E:m bigger than lambda(l)}).  We thus conclude that $v(c_l(T))=v(b_{l,m(l)}T^{m(l)})=m(l)v(T)$. This proves the claim.
\\

Now let $\Spc_{D,\omega}^{>\lambda}=\coprod_{i \geq 0} Y_{i,\omega}$ be the desired decomposition, and let 
\[
\tilde \alpha_0(\omega), \tilde\alpha_1(\omega), \dots
\] 
denote the sequence consisting of $\alpha_i$'s with multiplicity $\deg Y_{i,\omega}$. In the following, we will show that the sequence $\tilde \alpha_0(\omega), \tilde\alpha_1(\omega), \dots$ is a disjoint union of $p^{M-1}(p-1)t/2$ arithmetic progressions with the same common difference $\frac{\varphi(q)p^M}{2q^2}$.

Let $\psi$ be a character of conductor $p^M$.  We look at weights of the form $(k,\psi)$ for all $k\geq0$. First, note that  $v(T_{(k, \psi)})=\frac{q}{\varphi(p^M)}=\frac{q}{(p-1)p^{M-1}}$ by the assumption on $M$; thus $(k,\psi)\in\calW^{>\lambda}_{\psi|_{\Delta}\cdot\omega_0^k}$.
Noting the equality $\frac{q}{(p-1)p^{M-1}} \varphi(q) = q^2p^{-M}$, it then follows that the $U_p$-slopes of $S_{k+2}^D(K^p\Iw_{p^M}, \psi)$ are
\[
q^2p^{-M}\tilde \alpha_0(\psi|_{\Delta}\cdot\omega_0^k),\dots, q^2p^{-M}\tilde \alpha_{(k+1)q^{-1}p^{M}t-1}(\psi|_{\Delta}\cdot\omega_0^k).
\]
Hence, by Atkin--Lehner theory (Proposition~\ref{P:atkin-lehner}), in the $U_p$-slope sequence on $S_{k+2}^D(K^p\Iw_{p^M}; \psi^{-1})$, from the $(k q^{-1}p^{M}t+1)$st to the $(k+1)q^{-1}p^{M}t$th is given by
\[
k+1 - q^2p^{-M}\tilde\alpha_{qp^{M}t-1}(\psi|_{\Delta} \cdot \omega_0^k), \dots, k+1 - q^2p^{-M}\tilde\alpha_0(\psi|_{\Delta} \cdot \omega_0^k).
\]
This implies the relations
\begin{align*}
\tilde \alpha_{(k+1)q^{-1}p^{M}t -1-i}(\psi^{-1}|_{\Delta}\cdot\omega_0^{k}) = &q^{-2}p^{M}\big(k+1 - q^2p^{-M}\tilde\alpha_{i}(\psi|_\Delta \cdot \omega_0^k) \big)\\
=&(k+1)q^{-2}p^M-\tilde\alpha_{i}(\psi|_\Delta \cdot \omega_0^k)
\end{align*}
for $0\leq i\leq q^{-1}p^{M}t-1$.
Replacing $\psi$ by $\psi\omega_0^{-1}$ and $k$ by $k+1$, we get
\[
\tilde \alpha_{(k+2)q^{-1}p^{M}t-1-i}(\psi^{-1}|_{\Delta}\cdot\omega_0^{k+2})=(k+2)q^{-2}p^M-\tilde\alpha_{i}(\psi|_\Delta \cdot \omega_0^k). 
\] 
We thus deduce that
\begin{equation}
\label{E:arithmetic progression}
\tilde \alpha_{(k+2)q^{-1}p^{M}t-1-i}(\psi^{-1}|_{\Delta}\cdot\omega_0^{k+2})
=\tilde \alpha_{(k+1)q^{-1}p^{M}t-1- i}(\psi^{-1}|_{\Delta}\cdot\omega_0^{k})+q^{-2}p^M.
\end{equation}

We conclude the theorem by \eqref{E:arithmetic progression}. In fact, for any character $\omega$ of $\Delta$ and $j\in\ZZ_{\geq0}$, write  $j=(k+1)q^{-1}p^{M}t-1-i$ for some $k\in\ZZ_{\geq0}$ and $i\in[0, q^{-1}p^{M}t-1]$.  Choose $\psi$ so that 
$\psi|_{\Delta}\cdot\omega_0^{k}=\omega$. It then follows from (\ref{E:arithmetic progression}) that
\begin{equation}
\label{E:finer tilde alpha}
\tilde\alpha_{j+q^{-1}p^{M}t}(\omega \omega_0^2)=\tilde\alpha_{j}(\omega)+q^{-2}p^M. 
\end{equation}
In particular, since $\omega_0^{\varphi(q)} = \omega_0^{\frac{q(p-1)}{p}}=1$, we have
\[
\tilde\alpha_{j+(p-1)p^{M-1}t/2}(\omega)=\tilde\alpha_{j}(\omega)+\frac{\varphi(q)p^M}{2q^2}.
\]
Therefore, the sequence $\tilde \alpha_0(\omega),\tilde\alpha_1(\omega),\dots$ is the disjoint union of arithmetic progressions $\tilde \alpha_j(\omega),\tilde\alpha_{j+(p-1)p^{M-1}t/2}(\omega),\dots$ for
\[
0\leq j\leq(p-1)p^{M-1}t/2-1,
\] 
which have common difference $\frac{\varphi(q)p^M}{2q^2}$. 

\begin{remark}
The argument for the second part of the proof, namely, assuming the Claim to prove the slope ratios being the unions of arithmetic progressions, works equally well to the case of modular curves, as independently proved by Bergdall and Pollack \cite{bergdall-pollack}.
\end{remark}

\subsection{Proof of Corollary~\ref{C:classical}}
\label{S:proof of corollary}
(1) Specialize Theorem~\ref{T:main theorem} to the weight character $x^k\psi$
and note that $v(T_{x^k\psi}) = \frac{q}{(p-1)p^{m-1}}$ by the assumption on $m$.
If we use $\beta^\dagger_0(k, \psi), \beta^\dagger_1(k,\psi),\dots$ to denote the sequence of slopes of $U_p$-action on $S^{D, \dagger}_{x^k\psi}$,
then by (\ref{E:inequality1}) and (\ref{E:inequality2}) we have the following inequalities
\[
q^2p^{-m}
\lfloor n/qt\rfloor \leq
\beta^\dagger_n(k, \psi) \leq q^2p^{-m}\big( \lfloor n/qt \rfloor +1 \big) \quad \textrm{for all }n \geq 0.
\]
By the classicality result Proposition~\ref{P:classicality}, $\beta_i (k, \psi) = \beta^\dagger_i(k,\psi)$ for $i = 0, \dots, q^{-1}p^{m}(k+1)t-1$. This proves (1).

(2)\quash{First, pick any $p$-primitive character $\psi$ of $(\ZZ/p^{M}\ZZ)^\times$ such that $\psi|_{\Delta} = \omega$.}
Recall once again that $v(T_{\psi}) = \frac{q}{p^{M-1}(p-1)}$.
By specializing Theorem~\ref{T:main theorem 2} to the weight character $\psi$ that lifts $\omega$, and using the classicality result (Proposition~\ref{P:classicality}), we see that 
\[
\tilde \alpha_i(\omega) = p^Mq^{-2} \beta_i(\omega),
\]
for $i=0, \dots, q^{-1}p^{M}t-1$.
By \eqref{E:finer tilde alpha}, we have
\[
\tilde \alpha_{i+nq^{-1}p^{M}t}(\omega) = \tilde \alpha_i(\omega\omega_0^{-2n}) + np^Mq^{-2} =p^Mq^{-2}\beta_i(\omega\omega_0^{-2n}) + np^Mq^{-2}.
\]
Thus specializing Theorem~\ref{T:main theorem 2} to a general classical character $x^k\psi_m$ with $m \geq M$, we see the $U_p$-slopes on $S_{k+2}^{D, \dagger}(\psi_m)$ are exactly given by $ q^2p^{-m} \tilde \alpha_0(\psi_m|_\Delta  \omega_0^{k}), q^2p^{-m} \tilde \alpha_1(\psi_m|_\Delta  \omega_0^{k}), \dots$, or equivalently the set
\[
\bigcup_{n\geq 0}
 \big\{p^{M-m}(\beta_0(\psi_m|_\Delta  \omega_0^{k-2n})+n), \dots, p^{M-m}(\beta_{q^{-1}p^{M}t-1}(\psi_m|_\Delta  \omega_0^{k-2n})+n) \big\}.
\] 
By classicality result (Proposition~\ref{P:classicality}) again, we see that the slopes on $S^{D}_{k+2}(\psi_m)$ are those in the union with $n \in \{0, \dots, p^{m-M}(k+1)-1\}$. This concludes the proof of the corollary.

\section{Integral models of the space of overconvergent automorphic forms}
\label{Sec:integral model}

As mentioned before, the $U_p$-action on $S^{D}_\inte$ is unlikely to be compact.
This subtlety was carefully circumvented in the proof of our main theorem (e.g. the statement of Proposition~\ref{P:Up operator agrees}).
But we feel that it might be beneficial to introduce a variant construction, for which the $U_p$-action is compact.
We carry out this construction in this  section.

\begin{definition}
\label{D:compact operator}
Let $R$ be a complete noetherian ring, with ideal of definition $\gothm_R$.
Let $M$ be a topological $R$-module isomorphic to \[
\widehat \oplus_{i \in \ZZ_{\geq 0}} R e_i:=\varprojlim_{n} \Big(\bigoplus_{i \in \ZZ_{\geq 0}}(R/\gothm_R^n)e_i \Big),\] equipped with a continuous $R$-linear action of an operator $U$. We refer to $(e_i)_{i\in \ZZ_{\geq 0}}$ as an \emph{orthonormal basis}.
We say that the $U$-action on $M$ is \emph{compact} if the induced action on $M / \gothm^n_R M$ has finitely generated image for any $n \in \ZZ_{\geq 0}$.
This definition does not depend on the choice of the orthonormal basis of $M$.

When the $U$-action is compact, if $P$ denotes the infinite matrix for the $U$-action with respect to the basis $(e_i)_{i \in \NN}$, the \emph{characteristic power series} of the $U$-action:
\[
\Char(U; M): = \det (I_\infty - XP) = \lim_{n \to \infty} \det \big(I_\infty - X (P \textrm{ mod } \gothm_R^n) \big) \in R\llbracket X \rrbracket
\]
is well defined: note that its $r$-th coefficient is the trace of the action of $U$ on the $r$-th wedge product of $M$, which is well defined by first modulo $\gothm_R^n$ and then taking the limit. Moreover, it does not depend on the choice of the orthonormal basis.
\end{definition}

\begin{example}
\label{Ex:Up not compact}
We give an example where the operator is not compact; this example may be served as a toy model of the $U_p$-action on $S^{D}_\inte$. \quash{This is why we do not hope to prove the compactness of the $U_p$-action.}

Consider $M = \calC(\ZZ_p; \ZZ_p)$. The operator $U$ sends a continuous function $f(z)$ to
\[
h(z) = f(pz) + f(pz+1) + \cdots +f(pz+p-1).
\]
One can use Lemma~\ref{L:tilted degree estimate I} to control some of the entries of the infinite matrix of $U$ with respect to the Mahler basis.  But the $U$-action is not compact.
First note that the infinite matrix is going to be upper triangular because of the shape of the the operator $U$ and the trivial degree bound in Lemma~\ref{L:degree basic properties}(3).
Next, we look at the image of $\binom z {p^m}$ under $U$ for $m \geq 2$, which will appear on the $(p^m+1)$st column of the infinite matrix.
\[
U \bigg( \binom z{p^m} \bigg) = \binom{pz}{p^m}+ \binom{pz+1}{p^m} + \cdots +  \binom{pz+p-1}{p^m}.
\]
Note that evaluating the right hand side at $z = p^{m-1}$, we get $\sum_{i=0}^{p-1} \binom {p^m+i}i$, which is congruent to $p$ modulo $p^m$. On the other hand, it is clear that $U(\binom {z} {p^{m}})|_{z = i}$ is equal to zero for $i=0, 1, \dots, p^{m-1}-1$. It follows that the Mahler coefficient of $\binom z{p^{m-1}}$ is not divisible by $p^2$.  
This implies that the operator $U$ cannot be compact.
\end{example}

\begin{remark}
\label{R:Up almost compact}
\quash{As seen in the previous example, it is likely that the $U_p$-action on $S^{D}_\inte$ is not compact. }The non-compactness of $U_p$ may cause technical difficulties in applications.
Our fix to this problem is to introduce a subspace stable under the action of the monoid $\bfM_1$.
But we first explain that another apparently easier fix: developing a more general compact operator theory,  would not easily work.

As shown in Theorem~\ref{T:HP lower bound}, the $U_p$-action on $S^{D}_\inte$ satisfies the following property which is slightly weaker than being compact:
there exists an orthonormal basis such that, the associated infinite matrix,
modulo $\gothm_\Lambda^n$ for each $n$, is strictly upper triangular except for the first $d(n) \times d(n)$-minor for some $d(n) \in \NN$ depending on $n$.
It still makes sense to define characteristic power series for such an infinite matrix by taking the limit over the characteristic power series of its minors.
Unfortunately, this power series defined in this generality depends on the choice of the orthonormal basis (even if restricting to those bases satisfying the condition above).
Here is an example: consider $M = \widehat \bigoplus_{i \in \ZZ_{\geq 0}} \ZZ_p e_i$ equipped with the action of $U$, sending $e_0$ to $0$ and $e_i$ to $e_{i-1}$ for $i \in \NN$.
Then for this choice of orthonormal basis, the corresponding characteristic power series is just $1 \in \ZZ_p\llbracket X\rrbracket$, as the infinite matrix for $U$ is strict upper triangular.
Now if we consider another orthonormal basis of $M$:
\begin{align*}
e'_0 &= e_0 +pe_1 + p^2e_2 + p^3 e_3 + \cdots;
\\
e'_1 &= e_1 +pe_2 + p^2e_3 + p^3 e_4 + \cdots;
\\
e'_2 &= e_2 +pe_3 + p^2e_4 + p^3 e_5 + \cdots;
\\
&\cdots \quad \cdots
\end{align*}
Then we have $U(e'_0) = pe'_0$ and $U(e'_i) = e'_{i-1}$ for $i \in \NN$.  So the corresponding infinite matrix for $U$ is $p, 0, 0, 0, \dots$ on the main diagonal, all $1$ at the entries just above the diagonal, and $0$ everywhere else.
In particular, the corresponding power series is $1-pX \in \ZZ_p\llbracket X\rrbracket$.
So, in general, the characteristic power series for this type of operators might depend on the choice of the orthonormal basis.
\end{remark}

\subsection{Integral models of overconvergent automorphic forms}
\quash{It would be interesting to know if one can extend this construction to the entire weight space.
Recall that $\Lambda^{>1/p} = \ZZ_p\llbracket T, pT^{-1}\rrbracket \otimes_{\ZZ_p}\ZZ_p[\Delta]$.}
Let $[-]': \ZZ_p^\times \to (\Lambda^{>1/p})^\times$
denote the universal character of $\ZZ_p^\times$. Recall that \eqref{E:explicit induced representation} gives an isomorphism between $\Ind_{B(\ZZ_p)}^{\Iw_q}([-]')$
and $\calC(\ZZ_p; \Lambda^{>1/p})$; the latter admits an orthonormal basis (over $\Lambda^{>1/p}$) given by the functions $\big( \binom zn\big)_{n \in \ZZ_{\geq 0}}$.
We consider a closed subspace 
\begin{equation}
\label{E:modified subspace}
\Ind_{B(\ZZ_p)}^{\Iw_q}([-]')^{\textrm{mod}} = \widehat \bigoplus_{n \geq 0} T^n \Lambda^{>1/p} \cdot \tbinom zn.
\end{equation}
We claim that this subspace is stable under the action of the monoid $\bfM_1$.
Indeed, by Proposition~\ref{P:combined estimate}(2), for the action of $\delta_p \in \big(\begin{smallmatrix} a&b\\c&d \end{smallmatrix}\big) \in \bfM_1$ on the Mahler basis, the coefficient
$P_{m,n}(\delta_p)$ belongs to $
 \gothm_\Lambda^{\max\{m-n, 0\}}\Lambda^{>1/p} = T^{\max\{m-n, 0\}}\Lambda^{>1/p}$.
Then, with respect to the basis $\big(T^n \binom zn\big)_{n \in \ZZ_{\geq 0}}$, the $(m,n)$-entry of the infinite matrix has coefficients in
\[
T^{n-m}\cdot
T^{\max\{m-n, 0\}}\Lambda^{>1/p} = T^{\max\{0, n-m\}} \Lambda^{>1/p}.
\]
This concludes the proof of the claim.

Moreover, Proposition~\ref{P:combined estimate}(1) says that for $\delta_p \in \big(\begin{smallmatrix} a&b\\c&d \end{smallmatrix}\big) \in \big(\begin{smallmatrix} p\ZZ_p&\ZZ_p\\q\ZZ_p&\ZZ_p^\times\end{smallmatrix}\big)$, namely those $\delta_p$ appearing in the expression of $U_p$, the coefficient $P_{m,n}(\delta_p)$
belongs to $
\gothm_\Lambda
^{\max\{m-\lfloor n/p\rfloor, 0\}}$.
So under the new basis $\big(T^n \binom zn\big)_{n \in \ZZ_0}$, the $(m,n)$-entry of the infinite matrix has coefficients in
\[
T^{n-m}\cdot
\gothm_\Lambda^{\max\{m-\lfloor n/p\rfloor, 0\}}\Lambda^{>1/p} = T^{\max\{n-\lfloor n/p\rfloor, n-m\}} \Lambda^{>1/p} .
\]

Now, we define the space of \emph{integral $1$-overconvergent automorphic forms} to be
\[
S^{D, \dagger, 1}_{\textrm{int}}: = \Big\{ 
\varphi: D^\times \backslash (D \otimes \AAA_f)^\times / K^p \to \Ind_{B(\ZZ_p)}^{\Iw_q}([-]')^\textrm{mod}\; \Big|\;\varphi(xu_p) = \varphi(x)|\!|^{[-]}_{u_p}, \textrm{ for }u_p \in \Iw_q
\Big\}.
\]
It is a topological module over $\Lambda^{>1/p}$ isomorphic to  $\widehat \oplus_{i \in \ZZ_{\geq 0}} \Lambda^{>1/p}e_i$.
Viewing the $U_p$-action on $S^{D, \dagger, 1}_{\textrm{int}}$ with respect to the basis
\[
1_0, \dots, 1_{t-1}, T z_0, \dots, T z_{t-1}, T^2\tbinom  {z_0}2, \dots, T^2\tbinom{z_{t-1}}{2}, T^3\tbinom{z_0}3, \dots,
\]
the corresponding infinite matrix has its entry of its $n$th column in 
\[
T^{\lfloor n/t\rfloor -\lfloor n/pt\rfloor}\Lambda^{>1/p}.
\]
In particular, the action of $U_p$ is compact, and the characteristic power series $\Char(U_p; S^{D, \dagger, 1}_{\textrm{int}})$ agrees with the ones in Proposition~\ref{P:Up operator agrees}.\footnote{This construction was recently generalized by Johansson and Newton to general overconvergent cohomology \cite{johansson-newton}.}

\begin{remark}
\quash{Instead of taking the modified induced representation as in \eqref{E:modified subspace}, we could also take the subspace $\widehat \bigoplus_{n \geq 0} T^{\lfloor n/p\rfloor} \Lambda^{>1/p} \cdot \binom zn$ instead; or more generally using some appropriate sequence of numbers $a_0, a_1, \dots$ between $(n)_{n \geq 0}$ and $(\lfloor n/p\rfloor)_{n \geq 0}$.}
Similar constructions will give integral models of the space of $r$-overconvergent automorphic forms (with weights in $\Lambda^{>1/p}$) for $r>0$, on which the $U_p$-action is compact.
\end{remark}

\begin{remark}\label{R:final}
The above construction may be regarded as the \'etale realization of integral models of overconvergent automorphic forms. We are curious about the possibility of comparing our construction with the integral models constructed in \cite{AIP} by understanding the comparison theorem on this level.
\end{remark}


\begin{thebibliography}{999999}


\bibitem[AIP15${}^+$]{AIP}
F. Andreatta, A. Iovita, and V. Pilloni, {\it The adic Hilbert eigenvariety}, {\it to appear in the special volume for Coleman}, available at {\tt http://www.mat.unimi.it/users/andreat/research.html}

\bibitem[AIS14]{AIS}
F. Andreatta, A. Iovita, and G. Stevens, Overconvergent modular sheaves and modular forms for $\GL_{2/F}$. {\it Israel J. Math.} {\bf 201} (2014), no. {\bf 1}, 299--359.



\bibitem[Bel15${}^+$]{bellaiche}
J. Bellaiche, {\it Eigenvarieties and $p$-adic L-functions}. Book in preparation, available at {\tt http://people.brandeis.edu/\textasciitilde jbellaic/preprint/preprint.html}

\bibitem[BP15${}^+$]{bergdall-pollack}
J. Bergdall and R. Pollack,
Arithmetic properties of Fredholm series for $p$-adic modular forms, {\it to appear in Proc. of London. Math. Soc.}, {\tt  arXiv:1506.05307}.

\bibitem[BP16$^+$]{bergdall-pollack2}
J. Bergdall and R. Pollack,
Slopes of modular forms and the ghost conjecture, {\tt arXiv:1607.04658}.


\bibitem[Bu04]{buzzard2}
K. Buzzard,
On $p$-adic families of automorphic forms. In {\it Modular curves and abelian varieties}, 23--44, {\it Progr. Math.}, {\bf 224}, Birkh\"auser, Basel, 2004.

\bibitem[Bu05]{buzzard-question}
K. Buzzard,
Questions about slopes of modular forms. In {\it
Automorphic forms I}, {\it  
Ast\'erisque} no. {\bf 298} (2005), 1--15.

\bibitem[Bu07]{buzzard}
K. Buzzard,
Eigenvarieties. In {\it L-functions and Galois representations}, 59--120,
{\it London Math. Soc. Lecture Note Ser.}, {\bf 320}, Cambridge Univ. Press, Cambridge, 2007.

\bibitem[BG15${}^+$]{buzzard-gee}
K. Buzzard and T. Gee, Slopes of modular forms. {\it to appear in Proceedings of the 2014 Simons symposium on the trace formula}, {\tt arXiv:1502.02518}.

\bibitem[BK05]{buzzard-kilford}
K. Buzzard and L. Kilford,
The 2-adic eigencurve at the boundary of weight space.
{\it Compos. Math.} {\bf 141} (2005), no. 3, 605--619.

\bibitem[Ch05]{chenevier}
G. Chenevier,
Une correspondance de Jacquet--Langlands $p$-adique. {\it
Duke Math. J.} {\bf 126} (2005), no. {\bf1}, 161--194.


\bibitem[Cole-A]{coleman-halo}
R. Coleman, {\it private notes ``Halo"}.


\bibitem[CM98]{coleman-mazur}
R. Coleman and B. Mazur,
The eigencurve. In {\it Galois representations in arithmetic algebraic geometry (Durham, 1996)}, 1--113,
{\it London Math. Soc. Lecture Note Ser.}, {\bf 254}, Cambridge Univ. Press, Cambridge, 1998.

\bibitem[Colm10]{colmez}
P. Colmez,
Fonctions d'une variable $p$-adique. {\it Ast\'erisque} No. {\bf 330} (2010), 13--59.

\bibitem[DWX16]{davis-wan-xiao}
C.~Davis, D. Wan, and L. Xiao, Newton slopes for Artin--Schreier--Witt towers. {\it  Math. Ann.} {\bf 364} (2016), no. {\bf 3}, 1451--1468.


\bibitem[DL16]{diao-liu}
H. Diao and R. Liu,
The eigencurve is proper. {\it Duke Math. J.}, {\bf 165} (2016), no. {\bf 7}, 1381--1395.

\bibitem[Em98]{emerton}
M. Emerton,
$2$-adic modular forms of minimal slope. {\it Thesis at Harvard University}, 1998.

\bibitem[Em06]{emerton2}
M. Emerton,
On the interpolation of systems of eigenvalues attached to automorphic Hecke eigenforms. {\it Invent. Math.} {\bf 164} (2006), no. {\bf 1}, 1--84.

\bibitem[He05]{lisa-clay}
G. Herrick, Some conjectures about the slopes of modular forms. {\it Thesis at Northwestern University}, 2005, 48 pp.

\bibitem[Go88]{gouvea}F. Gouv\^ea,  Arithmetic of $p$-adic modular forms. {\it Lecture Notes in Mathematics} {\bf 1304}, Springer-Verlag, Berlin, 1988.

\bibitem[Hi02]{hida}
H. Hida, 
Control theorems for coherent sheaves on Shimura varieties of PEL-type. {\it Journal of the
Inst. of Math. Jussieu} {\bf1} (2002), no. {\bf 1}, 1--76.

\bibitem[Ja04]{jacobs}
D. Jacobs, Slopes of Compact Hecke Operators. {\it Thesis at University of London, Imperial College}, 2004.

\bibitem[JN$16^+$]{johansson-newton}
C. Johansson and J. Newton, Extended eigenvarieties for overconvergent cohomology. {\tt arXiv:1604.07739}.

\bibitem[Jo11]{jone}
O. Jones, An analogue of the BGG resolution for locally analytic principal series. {\it Journal of Number Theory} {\bf131} (2011), 1616--1640.

\bibitem[Ke09]{kedlaya}
K. Kedlaya,
Semistable reduction for overconvergent $F$-isocrystals, III: local semistable reduction at monomial valuations. {\it Compos. Math.} {\bf 145} (2009), 143--172.
 

\bibitem[Kil08]{kilford}
L. J. P. Kilford, On the slopes of the $U_5$ operator acting on overconvergent modular forms. {\it J. Th\'eor. Nombres Bordeaux} {\bf 20} (2008), no. {\bf 1}, 165--182.

\bibitem[KM12]{kilford-mcmurdy}
L. J. P. Kilford and K. McMurdy, Slopes of the $U_7$ operator acting on a space of overconvergent modular forms. {\it LMS J. Comput. Math.} {\bf 15} (2012), 113--139.

\bibitem[Kis03]{kisin} Mark Kisin, Overconvergent modular forms and the Fontaine--Mazur conjecture, {\it Invent. Math}. \textbf{153} (2003), no. 2, 373--454.

\bibitem[Lo11]{loeffler}
D. Loeffler, Overconvergent algebraic automorphic forms. {\it Proc. Lond. Math. Soc.} {\bf 102} (2011), no. {\bf 2}, 193--228.

\bibitem[LW12]{loeffler-weinstein}
D. Loeffler and J. Weinstein, 
On the computation of local components of a newform. {\it Mathematics of Computation} {\bf 81} (2012), 1179--1200. 

\bibitem[Pi13]{pilloni}
V.~Pilloni, Overconvergent modular forms. \textit{Ann. Inst. Fourier} \textbf{63} (2013), no. {\bf 1}, 219--239.

\bibitem[PX14$^+$]{pottharst-xiao}
J. Pottharst and L. Xiao,
On the parity conjecture in finite slope, {\tt arXiv:1410.5050}.

\bibitem[Ro14]{roe}
D. Roe, The $3$-adic Eigencurve at the boundary of weight space. {\it Int. J. Number Theory} {\bf 10} (2014), no. {\bf7}, 1791--1806.

\bibitem[ST02]{schneider-teitelbaum}
P. Schneider and J. Teitelbaum,
 Banach space representations and Iwasawa theory. {\it Israel J. Math.} {\bf 127} (2002), 359--380.

\bibitem[Wa98]{wan}
D. Wan,
Dimension variation of classical and $p$-adic modular forms. 
{\it Invent. Math.} {\bf133} (1998), 449--463.

\bibitem[WXZ14${}^+$]{wan-xiao-zhang}
D. Wan, L. Xiao, and J. Zhang, Slope of eigencurves over boundary disks, {\it to appear in Math. Ann.}, {\tt  arXiv:1407.0279}.
\end{thebibliography}
\end{document}